\documentclass[11pt]{amsart}

\usepackage{amsmath,amsthm}
\usepackage{amsfonts,amstext,amssymb, mathtools, comment}
\usepackage{amsfonts,amstext,amssymb,tikz}
\usepackage{graphicx,subcaption}
\RequirePackage[colorlinks,citecolor=blue,urlcolor=blue]{hyperref}

\newcommand{\levy}{L\'{e}vy }
\newcommand{\p}{{\mathbb P}}

\newcommand{\e}{{\mathbb E}}
\newcommand{\D}{{\mathrm d}}
\newcommand{\R}{{\mathbb R}}
\newcommand{\N}{{\mathbb N}}

\renewcommand{\a}{\alpha}

\newcommand{\ind}[1]{\mbox{\rm\large  1}_{\{#1\}}}

\newcommand{\F}{{\mathcal F}}

\newcommand{\stab}{\stackrel{d_{st}}{\longrightarrow}}

\newcommand{\n}{{(n)}}

\newcommand{\eqd}{{\stackrel{{\rm d}}{=}}}
\newcommand{\convd}{{\,\stackrel{{\rm d}}{\to}\,}}
\def \cip{\stackrel{\mathbb P}{\rightarrow}}

\newcommand{\mean}{\text{\rm mean}}
\newcommand{\med}{\text{\rm med}}

\newtheorem{theorem}{Theorem}
\newtheorem{proposition}{Proposition}
\newtheorem{corollary}{Corollary}
\newtheorem{lemma}{Lemma}
\newtheorem{remark}{Remark}

\mathtoolsset{showonlyrefs}

\begin{document}
	\title[Estimation of some random quantities of a L\'evy process]{Optimal estimation of some random quantities of a L\'evy process}
	\author[J.\ Ivanovs and M.\ Podolskij]{Jevgenijs Ivanovs and Mark Podolskij}
	\thanks{Jevgenijs Ivanovs gratefully acknowledges financial support of Sapere Aude Starting Grant 8049-00021B ``Distributional Robustness in Assessment of Extreme Risk''.}
	\thanks{Mark Podolskij gratefully acknowledges financial support of ERC Consolidator Grant 815703
		``STAMFORD: Statistical Methods for High Dimensional Diffusions''.}
	\address{Aarhus University, Denmark}

\begin{abstract}
In this paper we present new theoretical results on optimal estimation of certain random quantities based on high frequency observations of a \levy process. More specifically, we investigate the asymptotic theory for the conditional mean and conditional median estimators of the supremum/infimum of a linear Brownian motion and a stable \levy process. Another contribution of our article is the conditional mean estimation of the local time and the occupation time measure of a linear Brownian motion.  We demonstrate that the new estimators are considerably more efficient compared to the classical estimators studied in e.g. \cite{asmussen_glynn_pitman1995, B86, ivanovs_zooming, Jacod98, NO11}. 
Furthermore, we discuss pre-estimation of the parameters of the underlying models, which is required for practical implementation of the proposed statistics.
\end{abstract}

\maketitle

\section{Introduction} \label{sec1}

During the past decades the increasing availability of high frequency data in economics and finance has led to an immense progress 
in high frequency statistics. In particular, high frequency functionals of It\^o semimartingales  have received
a great deal of attention in the statistical and probabilistic literature, where the focus has been on estimation of quadratic variation, realised jumps and  related (random) quantities. 
A detailed discussion of numerous high frequency methods and their applications to finance can be found in the monographs \cite{AJ14,JP12}. 

Despite large amount of literature on high frequency statistics, the question of {\em optimality} has rarely been addressed. To fix ideas we consider 
a stochastic process $(X_t)_{t \in [0,1]}$ with a known law and an associated random quantity $Q=F((X_t)_{t \in [0,1]})$, where $F$ is a measurable functional. 
%Assume that we are given the observations $(X_{i/n})_{i\in[1:n]}$ over a uniform grid, where $[1:n]=\{1,\ldots,n\}$.
 The major problem of interest is outlined by the following question:
\begin{center}
{\em Given observations $(X_{i/n})_{i\in[0:n]}$, what is the optimal estimator 
of the random variable $Q$ and its asymptotic properties as $n\to\infty$?}
\end{center}   
Let us stress that we are interested in $Q$ for a particular realization of $(X_t)_{t \in [0,1]}$, which is observed over a dense grid, and not just in its law.

Of course, the formulated problem is hard to address in full generality. But even for particular model classes the assessment of optimality is far from trivial, which is mainly due to the randomness of $Q$. Indeed, the classical methods such as minimax theory, Le Cam theory or Cram\'er-Rao bounds, do not apply in this setting. There are only a few results in the literature that discuss optimality in high frequency statistics. In \cite{CDG13} the authors apply the infinite dimensional version of local asymptotic mixed normality to obtain lower efficiency bounds for estimation of integrated functionals of volatility in the setting of diffusion models with a particular structure.  In particular, their result shows that the standard estimator of the quadratic variation, the realised volatility, is indeed asymptotically efficient for the considered class of models. In a later paper \cite{CDG14} similar lower bounds have been obtained in the framework of certain jump diffusions. The paper \cite{NO11} discusses estimation of the occupation time measure for continuous diffusion models and the authors prove that $n^{3/4}$ is the optimal rate of convergence (however, they do not discuss efficiency bounds). The
articles \cite{A19a,A19b,AC18} investigate estimation of integral functionals $Q=\int_0^1 f(X_s) ds$ for various Markovian and 
non-Markovian models. The main focus here is on deriving error bounds and weak limit theorems for Riemann sum type estimators, which heavily depend on the smoothness of $f$. In several settings they also prove rate optimality in the case of Brownian motion.  

The aim of our paper is to study optimal estimation of extrema, local time and occupation time measure of certain \levy
processes. Accurate estimation of these random functionals is important for numerous applications. % applications including statistics and numerical analysis. 
For instance, supremum is a key quantity in insurance, queueing, financial mathematics, optimal stopping and various applied domains such as environmental science where maximal level of pollution is often of interest. It is noted that our theory can also be used in Monte Carlo simulation of extrema via discretization, but this is not our main focus since much better algorithms exist~\cite{cazares2018geometrically}; see also~\cite{mijatovic_suprema_stable} for exact simulation of the supremum of a stable process. These algorithms, however, can not handle, e.g., the diameter of the range of $X$, whereas our estimators still apply.
Accurate estimation of local times is required in a number of statistical methods including estimation of the volatility coefficient in a diffusion model \cite{FZ93}, estimation of the skewed Brownian motion \cite{L18} and estimation of the reflected fractional Brownian motion \cite{HL13}, just to name a few.  

The estimation of the aforementioned random quantities has been studied in several papers. The standard estimator of the 
supremum of a stochastic process is given by the maximum of its high frequency observations. In the setting 
of a linear Brownian motion the corresponding non-central limit theorem has been proven in \cite{asmussen_glynn_pitman1995}; their result has been later extended in \cite{ivanovs_zooming} to the class of \levy processes satisfying certain regularity assumption. Statistical inference for local times has been investigated in \cite{B86,Jacod98}, who showed asymptotic mixed normality for kernel type estimators in the framework of continuous SDEs. Finally,  \cite{AC18,NO11} discussed the estimation of the occupation time measure 
via Riemann sums. 

In this paper we show that the standard estimators proposed in the literature are indeed rate optimal, but they are not asymptotically efficient.
Instead of certain intuitive constructions, we consider the conditional mean and conditional median estimators, which turn out to be manageable in some important cases.
It is well known that the conditional mean $\e[Q|(X_{i/n})_{i\in[0:n]}]$ is the optimal $L^2$-predictor when $\e [Q^2]<\infty$.
In many cases considered below, however, the random variable $Q$ will not have a finite second moment. Then we use the conditional median estimator $\med[Q|(X_{i/n})_{i\in[0:n]}]$, which is optimal in $L^1$ sense given that $\e [|Q|]<\infty$.
Additionally, we still do consider the conditional mean which is a very natural estimator even when the second moment is infinite.
Importantly, it is optimal with respect to the Bregman distance: $D(x,y)=\phi(x)-\phi(y)-\phi'(y)(x-y)$ with $\phi$ being a  strictly convex differentiable function~\cite{banerjee}.
It is only required here that $\e [|Q|]$ and $\e [|\phi(Q)|]$ are finite. 
We often have $Q\geq 0$ and $\e [Q^p]<\infty$ for some $p>1$, and hence we may take $\phi(x)=x^p$ to produce an optimality statement for the conditional mean estimator. Finally, the conditional median is optimal with respect to
$D(x,y)=(\ind{x\geq y}-1/2)(g(x)-g(y))$ for an increasing function $g$ which in our case can be taken as $g(x)=x^p$ for $p>0$, see~\cite{gneiting} and references therein.

%Instead we consider the conditional mean and median estimators which are optimal in $L^2$ and $L^1$ sense, respectively, see also~\S\ref{sec:optimality} for further comments.
In the case of supremum, the conditional mean and median estimators have a rather explicit and simple form, but their performance assessment is not a trivial task. Importantly, self-similarity of $X$ (up to measure change) is the key property when evaluating such estimators and establishing the corresponding weak limit theory. 
Thus we consider the following two classes of processes: (i) linear Brownian motions and (ii) non-monotone self-similar L\'evy processes.
In the case of local/occupation time we only work with the class (i) of linear Brownian motions and focus on the conditional mean estimators exclusively, which is dictated by the structure of the problem and the tools currently available. Importantly, our conditional mean estimator of the local time fits the framework of~\cite{Jacod98} and yields an asymptotically optimal statistic in some large class in the case of continuous SDEs, see Remark~\ref{rem:localSDE}.  
We find that our new optimal estimators are considerably more efficient than the standard ones and that they do have narrower confidence intervals. In the case of supremum, this is illustrated by a numerical study. Furthermore, we discuss several modifications of our statistics including pre-estimation of unknown parameters of the underlying model. % which is either (i) or~(ii).

\bigskip
This paper is structured as follows. \S\ref{sec:sup_main} is devoted to the supremum and its conditional mean and median estimators with the corresponding weak limit theory in the case of a self-similar \levy process with a known law. Here we also treat the case of a linear Brownian motion, and comment on the conditional mean estimator of the range diameter.
In \S\ref{sec:BM_main} we present the conditional mean estimators of the local time and occupation time together with the asymptotic theory in the case of a linear Brownian motion.
Then in \S\ref{sec3} we study modified statistics based on pre-estimation of the unknown parameters of the model. In particular, we show that reasonable pre-estimation of the model parameters does not affect the asymptotic theory. Furthermore, the effect of truncation of the potentially infinite product involved in the construction of the supremum estimators is discussed, and some comments concerning a general \levy process are given.
Numerical illustrations for the case of supremum are presented in \S\ref{sec:numerics}, where both a linear Brownian motion and a one-sided stable processes are considered. The proofs are collected in Appendix~\ref{sec:proofs_sup} and Appendix~\ref{sec:proofs_loc} for the supremum and local/occupation time, respectively.
The former also requires some additional theory for \levy processes conditioned to stay positive which is given in Appendix~\ref{sec:cond_pos}. 

%will rather use the following statement: 
%\begin{align*}
%\med[Q|(X_{i/n})_{0\leq i \leq n}] \text{ is the } L^1 \text{-optimal predictor of } Q \text{ when } Q\in L^1(\p). 
%\end{align*}
%(Here med stands for median). 

%The following statement is a well known fact in the literature:
%\begin{align*}
%\e[Q|(X_{i/n})_{0\leq i \leq n}] \text{ is the } L^2 \text{-optimal predictor of } Q \text{ when } Q\in L^2(\p). 
%\end{align*}
%In other words, when $Q\in L^2(\p)$, $\e[Q|(X_{i/n})_{0\leq i \leq n}]$ is the minimiser of the expectation $\e[(Q-Y)^2]$, where the minimisation is performed over all $\sigma(X_{i/n}:~i=0,\ldots,n)$-measurable random variables $Y$. 

\section{Optimal estimation of supremum for a self-similar \levy process} \label{sec:sup_main}

In this section we assume that $(X_t)_{t \geq 0}$ is a non-monotone $1/\a$-self-similar L\'evy process, i.e.
\[
(X_{ut})_{t\geq 0}\stackrel{d}{=}u^{1/\a}(X_t)_{t\geq 0}\qquad \text{for all }u>0,
\]
where necessarily $\a\in(0,2]$. Assuming that the law of $X$ (or its parameters) is known, we focus on optimal estimation of the supremum and infimum of $X$ on the interval $[0,1]$ from high-frequency observations. 
%In this section we consider a \levy process $(X_t)_{t\geq 0}$ defined on a filtered probability space $\offp$. Moreover, we assume that the law of $X$ under $\p$, is known. 
The case $\a\in(0,2)$ corresponds to a strictly $\a$-stable process, whereas for $\a=2$ we have a scaled Brownian motion, and the respective simplified expressions for the statistics and their limits can be found in~\S\ref{sec:BM}.
In fact, \S\ref{sec:BM} considers a more general setting of a linear Brownian motion, which is not self-similar but becomes such under Girsanov change of measure.
Some further results concerning estimation of infimum and the range diameter are given in~\S\ref{sec:joint}.
%Otherwise, we have parameterized by two parameters.

We introduce
the notation 
\[
\overline X_t := \sup_{s\leq t} X_s \qquad \text{and} \qquad \underline X_t := \inf_{s\leq t} X_s
\]
to denote the running supremum and infimum process, respectively. 
Furthermore, the time of supremum will often be needed, and thus we define
\[\tau_t:=\inf\{s\in(0,t]:X_{s-}\vee X_s=\overline X_t\}.\]
In fact, the process $X$ as considered here does not jump at its supremum time almost surely and thus we could have used the term maximum instead.
The standard distribution free estimator of $\overline X_1$
is given by the empirical maximum of the observed data:
\begin{align} \label{Mm}
M_n:=\max_{i\in[0:n]} X_{i/n}. 
\end{align}
We remark, however, that $M_n$ is always downward biased. 
Finally, estimation of the infimum amounts to estimation of the supremum of $-X$, and thus no additional theory is needed. The joint estimation of supremum and infimum is discussed in~\S\ref{sec:joint}.

In the following we will often use the notion of {\em stable convergence}. We recall that a sequence of random variables $(Y_n)_{n\in \N}$ 
defined on $(\Omega, \mathcal F, \mathbb P)$ is said 
to converge stably  with limit $Y$ ($Y_n \stab Y$) defined on an extension
$(\overline \Omega, \overline{\mathcal F}, 
\overline{\mathbb P})$ of the original probability space $(\Omega, \mathcal F, \mathbb P)$, 
iff for any bounded, continuous function $g$ and any bounded $\mathcal{F}$-measurable random variable $Z$ it holds that
\begin{equation} \label{defstable}
\e[ g(Y_n) Z] \rightarrow \overline{\e}[ g(Y) Z], \quad \text{as } n \rightarrow \infty.
\end{equation}  
The notion of stable convergence is due to  Renyi \cite{REN}. We also refer to \cite{AE} for properties
of this mode of convergence.

\subsection{Preliminaries}
We will now review the asymptotic theory for the estimator $M_n$, which will be useful for studying conditional mean and median estimators. 
In order to state the limit theorem for $M_n$, we need to introduce an auxiliary process $(\xi_t)_{t\in\R}$.
It is defined as the following weak limit:
\begin{equation}\label{eq:xi}(\overline X_T-X_{\tau_T+t})_{t\in\R}\convd (\xi_t)_{t\in\R}\qquad\text{ as }T\to\infty,\end{equation}
see~\cite{bertoin_splitting}. Here and in the following it is tacitly assumed that the left hand side is $\infty$ when $\tau_T+t\notin[0,T]$.
The functional convergence is always with respect to the Skorokhod $J_1$ topology, unless specified otherwise.
It may be useful to think of $\xi$ as the process $X$ seen from its supremum as the time horizon tends to infinity.

It is well known that $(\xi_t)_{t\geq 0}$ and $(\xi_{(-t)-})_{t\geq 0}$ are independent finite Feller processes starting at $0$. Various representations of these processes exist and a number of important properties have been established, see e.g. \cite{chaumont} and references therein. The latter process when started at a positive level is often referred to as $X$ {\em conditioned to stay positive} (the negative of the former is $X$ {\em conditioned to stay negative}); here conditioning is understood in a certain limiting sense. The law of the limiting process $\xi$ is not explicit except when $X$ is a Brownian motion and then both parts of $\xi$ are $3$-dimensional Bessel processes scaled by $\sigma$, the standard deviation of $X_1$.
In all cases $\xi$ inherits self-similarity from $X$, and hence both parts (when started from positive values) are positive self-similar Markov processes admitting Lamperti representation studied in detail in~\cite{caballero2006conditioned}.

Due to self-similarity of the process $X$ it holds that
\begin{equation}\label{eq:zooming}
\xi^\n_t:=n^{1/\a}\left(\overline X_1- X_{\tau_1 + \frac{t}{n}} \right)_{t\in \mathbb{R}} \convd \left(\xi_t \right)_{t\in \mathbb{R}}\qquad \text{ as }n\to\infty,
\end{equation}
where again $\xi_t^\n=\infty$ when $\tau_1 + \frac{t}{n}\notin[0,1]$.
In other words, the process $\xi$ arises from zooming-in on $X$ at its supremum point.
We refer the reader to~\cite{asmussen_glynn_pitman1995,ivanovs_zooming} for the case of a linear Brownian motion and a general L\'evy process, respectively.

The following result is an instructive application of the convergence in~\eqref{eq:zooming}. It is a particular case of~\cite[Thm.\ 5]{ivanovs_zooming} extending the result of~\cite{asmussen_glynn_pitman1995} for Brownian motion.
\begin{theorem} \label{Mntheo} %\cite[Theorem 5]{ivanovs_zooming}
For a non-monotone $1/\a$-self-similar L\'evy process $X$ we obtain  the stable convergence as $n\to\infty$:
\begin{equation} \label{Maxconv}
V^\n:=n^{1/\a} (\overline X_1-M_n) \stab V:=\min_{j\in \mathbb{Z}} \xi_{j+U}
\end{equation}
where $\xi$ and the standard uniform $U$ are mutually independent, and independent of~$\F$.   
\end{theorem}
Let us mention the underlying intuition, which will be important to understand our main result in Theorem~\ref{thm:max} given below. 
Note the identity
\begin{align} \label{Mndec}
n^{1/\a} (\overline X_1-M_n) = \min_{j\in \mathbb Z} \xi^\n_{j+\{n\tau_1\}}
\end{align}
where $\{x\}$ stands for the fractional part of $x$. The random time $\tau_1$ has a density~\cite{C13} and thus according to~\cite{JP12,K37}
\[\{n\tau_1\} \stab U,\]  which together with~\eqref{eq:zooming} hint at~\eqref{Maxconv}.
It is noted that the convergence in~\eqref{eq:zooming} is, in fact, stable with $\xi$ being independent of~$\mathcal F$. Intuitively, zooming-in at the supremum makes the values of $X$ at some fixed times irrelevant.
We stress that this only provides intuition and the proof is far from being complete, see~\cite{ivanovs_zooming} and also~\cite{ivanovs_firstpassage} providing the necessary corrections.

\subsection{Optimal estimators}\label{sec:opt_estimators_sup}
Let us proceed to construct our optimal estimators given by the conditional mean and median. For this purpose we introduce the conditional 
distribution of $\overline X_1$ given the terminal value $X_1$ via
\begin{align} \label{Fxy}
F(x,y):=\p(\overline X_1 \leq x| X_1=y).  
\end{align} 
We choose a version continuous in $y$ which is, in fact, jointly continuous in $(x,y)$ as will be shown in Lemma~\ref{lem:F} below.
By self-similarity we also have
\[F_{1/n}(x,y):=\p(\overline X_{1/n} \leq x| X_{1/n}=y)=F(n^{1/\a}x, n^{1/\a}y).\]
Next, consider the conditional distribution of $\overline X_1-M_n$ given the observations:
\begin{align*} %\label{Hn}
H_n(x)&:=\p\left(\overline X_1-M_n\leq x| X_{j/n},\, j\in[1:n]\right)\\
&=\prod_{j=0}^{n-1}F_{1/n}\left(x+M_n-X_{j/n},X_{\frac{j+1}{n}}-X_{\frac{j}{n}})\right)\\
&=\prod_{j=0}^{n-1}F\left(n^{1/\a}(x+\Delta^n_j),n^{1/\a}(\Delta_{j}^n-\Delta_{j+1}^n)\right)\qquad \text{ for all }x\geq 0,
\end{align*}
where $\Delta_j^n:=M_n-X_{j/n}$ and the second line follows from the stationarity and independence of increments.
We note that $H_n(x)$ is continuous and strictly increasing in $x\geq 0$. % with $H_n(0)=0$ and $H_n(\infty)=1$.
Finally, we introduce the conditional mean and conditional median estimators of $\overline X_1$:
\begin{align} \label{eq:Tmean_def}
&\overline T_n^\mean:=\e[\overline X_1| X_{j/n},\, j\in[1:n]]=M_n+\int_0^\infty (1-H_n(x))\D x,\\
&\overline T_n^\med:= \med[\overline X_1| X_{j/n},\, j\in[1:n]]=M_n+H^{-1}_n(1/2),\label{eq:Tmed_def}
\end{align}
where in the first line we use the integrated tail formula.
Interestingly, $\overline T_n^\mean<\infty$ even when $\e \overline X_1=\infty$, see Remark~\ref{rem:finiteness}.
When evaluating our statistics defined in~\eqref{eq:Tmean_def} and \eqref{eq:Tmed_def} we need access to the function $F(x,y)$.
This function, however, is explicit only in the Brownian case analyzed in~\S\ref{sec:BM} and is semi-explicit in the case of one-sided jumps, see Proposition~\ref{prop:F}.
Thus, in the case of general strictly stable process one needs to assess $F$ numerically, which may necessitate truncation of the product in the definition of $H_n$.
Such modifications are discussed in \S\ref{sec:simplified}.

\subsection{Limit theory}
We start by noting that $H_n \convd \delta_{\overline X_{1}}$ $\p$-almost surely, whereas 
$H_n(xn^{-1/\a})$ has a non-trivial limit.
Observe that $\xi^\n_{j+\{n\tau_1\}}$ 
is the rescaled distance of the $j$th observation following $\tau_1$ from the supremum.
Thus 
\[H_n(xn^{-1/\a})=  \prod_{j\in \mathbb{Z}}F\left(x+\xi^\n_{j+\{n\tau_1\}}-V^\n,\xi^\n_{j+\{n\tau_1\}}-\xi^\n_{j+1+\{n\tau_1\}}\right),\]
where we tacitly assume that the factors with $\xi^\n_\cdot=\infty$ evaluate to~$1$. 
In view of Theorem \ref{Mntheo} it is intuitive that the limit is
\begin{align} \label{defH}
H(x):=  \prod_{j\in \mathbb{Z}}F\left(x+\xi_{j+U}-V,\xi_{j+U}-\xi_{j+1+U}\right),
\end{align}
where the random quantities $U,\xi$ and $V$ are defined in  Theorem \ref{Mntheo}. By substitution we obtain the identities
\begin{align} \label{eq:Tmean_id}
&\overline T_n^\mean=M_n+n^{-1/\a}\int_0^\infty (1-H_n(n^{-1/\a}x))\D x,\\
&\overline T_n^\med =M_n+n^{-1/\a}\ H_n(n^{-1/\a} \cdot)^{-1}(1/2),\label{eq:Tmed_id}
\end{align}
which suggest the asymptotic behaviour of our estimators defined in~\eqref{eq:Tmean_def} and~\eqref{eq:Tmed_def}.
We formalise this in one of our main results:
\begin{theorem} \label{thm:max}
	Assume that $X$ is a non-monotone $1/\a$-self-similar  L\'evy process.
	Then the random function $H$ is continuous and strictly increasing with $H(0)=0$ and $H(\infty)=1$ $\p$-a.s. and 
	\begin{equation}\label{eq:Hconv}
	\left(n^{1/\a}(\overline X_1-M_n),(H_n(xn^{-1/\a}))_{x\geq 0}\right)\stab (V,(H(x))_{x\geq 0})
	\end{equation}
	with respect to the uniform topology, where $V$ and $H(x)$ are defined in~\eqref{Maxconv} and ~\eqref{defH}, respectively.
Furthermore, our estimators satisfy
\begin{align} 
n^{1/\a}(\overline X_1-\overline T_n^\mean)  &\stab V -\int_0^\infty (1-H(x))\D x,\quad \text{ when }\a\in(1,2],\label{eq:T1}\\
n^{1/\a}(\overline X_1-\overline T_n^\med)&\stab V-H^{-1}(1/2),\label{eq:T2}
\end{align}
where the limit random variables are finite. 
\end{theorem}
It is noted that the proof of this result is far from trivial, since it requires precise understanding of the tail function $1-F(x,y)$ for large $x$
and the rate of growth of $\xi^\n_t$ as $t\to\infty$ (uniformly in $n$) among other things.  
%With exception of the Brownian case treated in~\ref{sec:BM_main} the resulting limiting expressions are non-trivial to simulate.
The identities \eqref{eq:Tmean_id} and \eqref{eq:Tmed_id} show that the statistics $\overline T_n^\mean$ and $\overline T_n^\med$ are first order 
equivalent to the standard estimator 
$M_n$, and the knowledge of the distribution of $X$ only enters through the $n^{-1/\a}$-order term.  This fact will prove to be important in Section \ref{sec3}, where the parameters of the law of $X$ will need to be estimated.  

Recall that $\e \overline X_1^p<\infty$ for $p\in(0,\a)$. Moreover, all moments of $\overline X_1$ are finite when $X$ is a Brownian motion or a strictly $\a$-stable process with no positive jumps. 
In the latter cases the conditional mean estimator is optimal in $L^2$ sense. In the case $\a\in(1,2]$ the conditional median is optimal in
$L^1$ sense and the conditional mean is optimal with respect to the above mentioned Bregman distance $D(x,y)=x^p-y^p-py^{p-1}(x-y)$, where $p\in(1,\alpha)$. 
Finally, the conditional median is optimal with respect to the loss function $D(x,y)=(\ind{x\geq y}-1/2)(x^p-y^p)$ for $p\in(0,\alpha)$ and any~$\alpha$.

Interestingly, all the expressions in Theorem~\ref{thm:max} stay the same if the process $X$ is replaced by its negative $-X$, see Proposition~\ref{prop:symmetry}. In particular, in the spectrally-positive case the difference $\overline X_1-\overline T_n^\mean$ has moments of all orders even though each term has infinite second moment, see also Remark~\ref{rem:finiteness} below.

%We recall that $\overline T_n^\mean$ is an $L^2$-optimal estimator for $\overline X_1$ if $\a=2$, while $\overline T_n^\med$ is an $L^1$-optimal estimator for $\overline X_1$ when $\a \in (1,2]$. 
%To understand the basic ideas behind the main results of Theorem \ref{thm:max}, we observe the identity
%\begin{align*}
%\overline T_n^\mean&=M_n+\int_0^\infty (1-H_n(x))\D x \\
%& = M_n+n^{-1/\a}\int_0^\infty (1-H_n(yn^{-1/\a}))\D y,
%\end{align*}
%and similarly 
%\begin{align*}
%\overline T_n^\med=M_n+H^{-1}_n(1/2) = M_n+n^{-1/\a}H_n(n^{-1/\a} \cdot)^{-1}(1/2).
%\end{align*}
%Taking into account the stable convergence at \eqref{eq:Hconv}, the statements of \eqref{eq:T1} and \eqref{eq:T2} appear to be a natural consequence.  However, 

\subsection{Linear Brownian motion}\label{sec:BM}
Consider a linear Brownian motion $X$ with drift parameter $\mu\in\R$ and scale parameter $\sigma>0$, which is
self-similar (and hence Theorem~\ref{thm:max} applies) only when $\mu=0$.
Nevertheless, $X$ can be obtained from a scaled Brownian motion by Girsanov change of measure and, in particular, the conditional distribution $\p(\overline X_{1/n}\leq x|X_{1/n}=y)$ does not depend on~$\mu$, see \S\ref{secA.4.1}. Hence our estimators have exactly the same form as in the case of $\mu=0$, see~\S\ref{sec:opt_estimators_sup}. Furthermore, the conditional distribution function $F$ is explicit in this case and is given by
\begin{equation}\label{eq:F_BM}
F(x,y)=1-\exp\left(-2x(x-y)/\sigma^2\right)\qquad \text{for }x>y_+,
\end{equation}
which follows from~\cite{shepp} or earlier sources, see also~\cite[1.1.8]{borodin_salminen}. 
Thus
\begin{equation}\label{eq:H_BM}
H_n(x)=\prod_{i=0}^{n-1}\left(1-\exp(-2(x+\Delta_i)(x+\Delta_{i+1})n/\sigma^2)\right)
\end{equation}
and the estimators are then defined by~\eqref{eq:Tmean_def} and~\eqref{eq:Tmed_def}. 

Interestingly, also the limit theorem has exactly the same form. The main reason for this is that the limit in~\eqref{eq:zooming} does not depend on $\mu$ either, see~\cite{asmussen_glynn_pitman1995}.
In the following result we prefer to choose the scaling $\sqrt n/\sigma$ rather than $\sqrt n$ so that the respective quantities correspond to the standard Brownian motion. % and, in particular, $\xi$ is the two-sided 3-dimensional Bessel process.
\begin{corollary}\label{cor:BM}
For a linear Brownian motion $X$ with drift parameter $\mu$ and scale $\sigma>0$ we have
\begin{align}
\frac{\sqrt n}{\sigma}(\overline X_1-\overline T^\mean_n)  &\stab V -\int_0^\infty (1-H(x))\D x,\label{eq:BM1}\\
\frac{\sqrt n}{\sigma} (\overline X_1-\overline T^\med_n)&\stab V-H^{-1}(1/2),\label{eq:BM2}
\end{align}
where $V=\min_{j\in \mathbb{Z}} \xi_{j+U}$ and 
\[H(x)=\prod_{j\in \mathbb{Z}}\left(1-\exp\left(-2(x+ \xi_{j+U}-V )(x+ \xi_{j+1+U}-V) \right)\right)\]
with $\xi$ being the two-sided 3-dimensional Bessel process and $U$ a standard uniform, which are mutually independent and independent of~$\mathcal F$.
\end{corollary}
 %here we slightly abuse the notation by assuming that $\xi$ and $H(x)$ correspond to the case of the standard Brownian motion.
% The proof requires an additional change of measure argument, see [TODO]. 

Additionally, we show that~\eqref{eq:BM1} extends to convergence of moments, see Lemma~\ref{lem:UI} below.
In particular, the asymptotic MSE of the optimal $\overline T^\mean_n$ is given by
\[\e [(\overline X_1-\overline T^\mean_n)^2]\sim\frac{\sigma^2}{n}\e\left[\left(V-\int_0^\infty (1-H(x))\D x\right)^2\right].\]

 \begin{lemma}\label{lem:UI}
 	For a linear Brownian motion $X$ and any $p>0$ we have
 	\[\e\left[\left(\frac{\sqrt n}{\sigma}\left(\overline X_1-\overline T_n^\text{\rm mean}\right)\right)^p\right]\to \e\left[\left(V-\int_0^\infty (1-H(x))\D x\right)^p\right]<\infty.\]
 \end{lemma}
%nonsense below?
%\begin{remark}\label{rem:2moment}\rm
%	It is shown in~\cite{asmussen_glynn_pitman1995} that $\e V_n^2\to \e V^2<\infty$, whereas
%$n\e(\overline X_1-\overline T_n^\mean)^2\leq \e V_n^2$ by the optimality of the conditional mean. Hence the limit random variable in~\eqref{eq:BM1} has a finite second moment. It does not imply, however, that the second moments in~\eqref{eq:BM1} converge to the respective quantity.
%\end{remark}

\subsection{Joint estimation of supremum and infimum}\label{sec:joint}
Consider the process $-X_t$ and the associated conditional mean estimator $\underline T^\mean_n$ of its supremum $\sup_{t\in[0,1]}(-X_t)=-\underline X_1$,
which is the negative of the infimum of~$X$.
According to Proposition~\ref{prop:symmetry} there is the symmetry:
\[(-\underline X_1)-\underline T^\mean_n\stackrel{d}{=}\overline X_1-\overline T^\mean_n\]
for all~$n$, and so also the asymptotic theory is the same.
Furthermore, we have the following joint convergence (linear Brownian motion included with $\a=2$ and then the limit corresponds to the case $\mu=0$):
\begin{corollary}\label{cor:joint}
For $\a\in(1,2]$ it holds that
\begin{align*}n^{1/\a}(\overline X_1-\overline T^\mean_n,-\underline X_1-\underline T^\mean_n) \stab (L,L'),\end{align*}
where $L'$ and $L$ are identically distributed, mutually independent, and independent of $\F$. Their common distribution is the limiting law in~\eqref{eq:T1}.
\end{corollary}
This, for example, readily yields the limit result for the conditional mean estimator of the range diameter $\overline X_1-\underline X_1$.

%
%Finally, we discuss the joint stable convergence for the estimators of $\overline X_1$ and $\underline X_1$. Similarly to \eqref{eq:Tmean_def} we define
%\[
%\underline{T}_n^{\mean}= \e[\underline X_1| X_{j/n},\, j\in[1:n]]
%\]
%We obtain the following result.
%\begin{corollary}
%We denote by $\mu_1$ (resp. $\mu_2$) the limiting distribution at \eqref{eq:T1} (resp. \eqref{eq:T2}). We consider random variables 
%$(L_1,L_1^{\prime}) \sim \mu_1 \otimes \mu_1$ and $(L_2,L_2^{\prime}) \sim \mu_2 \otimes \mu_2$, defined on an extension of the probability
%space $\ofp$, that are independent of $\F$. Then it holds that 
%\begin{align*} 
%n^{1/\a}\left(\overline X_1-\overline T_n^\mean, \underline X_1-\underline T_n^\mean  \right)  &\stab (L_1,L_1^{\prime}),\qquad \text{ when }\a\in(1,2],\\
%n^{1/\a}\left(\overline X_1-\overline T_n^\med, \underline X_1-\underline T_n^\med  \right)  &\stab (L_2,L_2^{\prime}). 
%\end{align*}
%\end{corollary}
%This corollary can be applied to the estimation of the maximal span $\overline X_1 - \underline X_1$. 

\section{Optimal estimation of local time and occupation time measure for a linear Brownian motion} \label{sec:BM_main}
In this section $X$ denotes a linear Brownian motion with drift parameter $\mu\in\R$ and scale $\sigma>0$, and $L_t(x)$ denotes the corresponding local time process at the level~$x\in\R$, which is a continuous increasing process given as the almost sure limit:
\begin{equation*}L_t(x):=\lim_{\epsilon\downarrow 0}\frac{1}{2\epsilon}\int_0^t 1_{(x-\epsilon,x+\epsilon)}(X_s)\D s.\end{equation*}
Furthermore, $O_t(x)$ stands for the occupation time in the interval $(x,\infty)$:
\begin{equation}\label{eq:occ_density}
O_t(x):= \int_0^t 1_{(x,\infty)} (X_s) ds = \int_x^\infty L_t(y)\D y \quad\text{a.s.}
\end{equation}
Our aim here is to establish limit theorems for the conditional mean estimators of $L_t(x)$ and $O_t(x)$.

\subsection{Basic formulae}
An important role will be played by the functions
\begin{align*}g(x,z)&:=\e^0[L_1(x)| X_1=z],\\
 G(x,z)&:=\e^0[O_1(x)| X_1=z]=\int_x^\infty g(y,z)\D y,
 \end{align*}
where $\e^0$ corresponds to the law of the standard Brownian motion. 
Both functions $g$ and $G$ have explicit formulae in terms of the density $\varphi$ and survival function $\overline \Phi$ of the standard normal distribution.
Some basic observations and these formulae are collected in the following result.
\begin{lemma}\label{lem:formulae} There are the identities
	\begin{align}
	&\e\left[L_t(x)| X_{t}=z\right] = \frac{\sqrt t}{\sigma} g\left(\frac{x}{\sigma\sqrt t},\frac{z}{\sigma\sqrt t}\right),\\
	&\e\left[O_t(x)| X_{t}=z\right] = t G\left(\frac{x}{\sigma\sqrt t},\frac{z}{\sigma\sqrt t}\right).
	\end{align}
Moreover, the functions $g$ and $G$ are bounded on $\R^2$ and satisfy $g(x,z)=g(-x,-z), G(x,z)=1-G(-x,-z)$.
For $x\geq 0$ we have the formulae
\begin{align*}
z<x: &\qquad g(x,z)=\overline\Phi(2x-z)/\varphi(z),\\
&\qquad G(x,z)=\frac{1}{2}\exp(-2x(x-z))-(2x-z) \frac{\overline\Phi(2x-z)}{2\varphi(z)},\\
z\geq x: &\qquad g(x,z)=\overline\Phi(z)/\varphi(z),\\ &\qquad G(x,z)=\frac{1}{2}+(z-2x) \frac{\overline\Phi(z)}{2\varphi(z)}.
\end{align*}
\end{lemma}

\subsection{Estimators and the limit theory}
The conditional mean estimators of $L_t$ and $O_t$ are easily derived using stationarity and independence of increments of $X$ together with Lemma~\ref{lem:formulae}:
\begin{align} \label{Lhat}
\widehat{L}_t(x)&=\e[L_t(x)|(X_{i/n})_{i\geq 1}] \\
&= \frac{1}{\sigma\sqrt n} \sum_{i=1}^{\lfloor nt\rfloor} g\left(\frac{\sqrt n}{\sigma}(x-X_{\frac{i-1}{n}}), \frac{\sqrt n}{\sigma}\Delta_i^n X\right)+O_\p(n^{-1/2}), \nonumber\\
\label{Ohat} \widehat O_t(x) &= \e[O_t(x)|(X_{i/n})_{i\geq 1}] \\&= \frac 1n \sum_{i=1}^{\lfloor nt \rfloor } 
G \left(\frac{\sqrt n}{\sigma} (x-X_{\frac{i-1}{n}}), \frac{\sqrt n}{\sigma}   \Delta_i^n X \right)
+O_\p(n^{-1}),\nonumber
\end{align}
where $\Delta_i^n X=X_{\frac{i}{n}}-X_{\frac{i-1}{n}}$. It is noted that the lower order terms can be written down explicitly (they are $0$ when $tn$ is an integer), but we keep them implicit, because they do not have an influence on the limit theorem presented below. 
%Finally, boundedness of this term follows from that of~$f$.

\begin{theorem}\label{thm:loc}
	Assume that $X$ is a linear Brownian motion with drift parameter $\mu \in \R$ and scale $\sigma>0$.
	 Then for any $x\in \R$ we have the functional stable convergence:
	\begin{align}\label{eq:loc_thm}
	n^{\frac{1}{4}} \left(\widehat{L}_t (x)- L_t (x) \right) &\stab \frac{v_l}{\sqrt \sigma}W_{L_t(x)},\\
	n^{\frac 34} \left(\widehat O_t(x) -  O_t(x) \right) &\stab v_o\sqrt\sigma W_{L_t(x)} \label{clton},
	\end{align}
	where $W$ is a Brownian motion independent of $\F$ and 
	\begin{align*}v_l^2&=\int_{\R} \e^0\left[g(y, X_1) - L_1(y)\right]^2 \D y=2\frac{3\log(1+\sqrt 2)-\sqrt 2}{3\sqrt\pi}\approx 0.4626,\\
	v_o^2&= \int_{\R}\e^0[G(y,X_1) - O_1(y)]^2\D y=\frac{13\sqrt 2-15\log(1+\sqrt 2)}{45\sqrt \pi}\approx 0.065.
	\end{align*}
	%\int_{\R} \e_{\mu=0}\VAR[L_{[0,1]}(y)|X_1] \D y=\int_{\R} \left(\e[L_{[0,1]}(y)|X_1]- L_{[0,1]}(y)\right)^2 \D y
\end{theorem}

Importantly, our conditional mean estimator~\eqref{Lhat} is a particular example of a more general class of statistics investigated in~\cite{Jacod98}
in the context of continuous diffusion processes. 
The expression for $v_l$  in~\cite{Jacod98} is rather lengthy and hard to evaluate, because of the generality assumed therein. In our case, $g(x, X_1) =\e[L_1(x)|X_1]$ is the conditional expectation and, in fact, a rather short direct proof can be given yielding the constant $v_l^2$ at the same time, see Appendix~\ref{sec:proofs_loc}.

\begin{remark}\label{rem:loc_alt}\rm
The above $v_l^2$ can be compared to $\frac{3}{3\sqrt\pi}(\sqrt 2-1)\approx0.6232$ obtained when instead of the optimal $g(x,z)$ one uses the kernel $\hat g(x)=\int_\R(|x+u|-|x|)\varphi(u)\D u$ depending on $x$ only, see~\cite[(1.27)]{Jacod98}.
The corresponding estimator (for $\sigma=1$) is $\frac{1}{\sqrt n} \sum_{i=1}^{\lfloor nt\rfloor} \hat g(\sqrt n(x-X_{\frac{i-1}{n}}))$, which does not take the increment following $X_{\frac{i-1}{n}}$ into account.
%Intuitively, this increment is important since the inter-observation local time is 
%It is noted that RMSE of the optimal method is about $86\%$ of the one based on $\hat g$.
\end{remark}

\begin{remark} \label{rem:localSDE} \rm
Consider the class of continuous SDEs defined via the equation
\[dX_t = \mu(X_t)dt + \sigma(X_t) dB_t,\]
where $B$ is a standard Brownian motion and $\sigma \in C^1(\R), \mu \in C(\R)$ are such that the above SDE has a unique strong solution. 
In~\cite{Jacod98} the author considers statistics of the form 
\[
L(h;x)_t^n= 
\frac{1}{\sqrt n} \sum_{i=1}^{\lfloor nt\rfloor} h\left(\sqrt n(x-X_{\frac{i-1}{n}}), \sqrt n \Delta_i^n X\right). 
\]
When $\sigma>0$ and $|h(y,z)|\leq \tilde{h}(y) \exp(a|z|)$ with $\tilde{h}$ bounded and satisfying $\int_{\R} |y|^r\tilde{h}(y) \D y <\infty$
for some $r>3$, 
the stable convergence 
\[
n^{1/4} \left( L(h;x)_t^n - c_h(x)L_t(x)\right) \stab v_h(x) W_{L_t(x)}
\]
holds, see~\cite[Theorem 1.2]{Jacod98}. Furthermore, the positive constant $v_h(x)$ (and the proof of stable convergence) 
stems from the simpler model $X_t=\sigma(x) B_t$. Hence, we can conclude that our estimator $\widehat{L}_t (x)$ is asymptotically optimal within 
the class of statistics $L(h;x)_t^n$ in the general setting of continuous SDEs. We believe that the restriction to the class $L(h;x)_t^n$ is not required
and $\widehat{L}_t (x)$ is asymptotically efficient for continuous SDEs. Furthermore, when the function $\sigma$ is unknown the coefficient 
$\sigma(x)$ can be estimated with a $n^{1/3}$-accuracy \cite{FZ93} and we can build a feasible statistic without affecting the asymptotic theory (cf. 
Proposition \ref{prop:cip0} below). 
\end{remark}

\section{Some modifications of the proposed statistics} \label{sec3}
%The scenario where parameters need to be estimated is considered in \S\ref{sec:params_max} and \S\ref{sec:params_loc} for the supremum and local/occupation time, respectively
The main goal of this section is to show that the above developed theory also applies in the setting when the law of $X$ is not known, but a consistent estimator of the parameters is available.
Furthermore, we construct certain simplified estimators of the supremum in order to cope with potential numerical issues.
%Finally, we give some comments concerning possible genera
 
\subsection{Unknown parameters}\label{sec:params_loc}
The main results of Theorem~\ref{thm:max} and Theorem~\ref{thm:loc} above assume that the law of the process $X$ is known, which is hard to accept in practice.
At most, we are willing to assume that the process $X$ belongs to some parametric class, and we distinguish between the following two: 
\begin{itemize}
	\item[(i)] Linear Brownian motion with drift parameter $\mu \in \R$ and scale $\sigma>0$, where for convenience we set $\a=2$. 
	As we remarked earlier neither the statistics nor the limits in Corollary \ref{cor:BM} and Theorem \ref{thm:loc} depend on $\mu$,
	which, in fact, can not be estimated consistently. Hence, the only 
	parameter of interest is $\theta=\sigma$. 
\item[(ii)] Non-monotone self-similar L\'{e}vy process which is naturally parameterized~\cite[\S I.5]{zolotarev_book} by a triplet $\theta=(\a,\rho,\lambda)$, where $\rho=\p(X_1>0)$ is the positivity parameter and $\lambda=\e[\log(|X_1|)]$ is related to the scale.
It is noted that $\rho\in[1-1/\a,1/\a]$ for $\a\in(1,2]$, and $\rho\in(0,1)$ for $\a\in(0,1]$ which excludes monotone processes.
This parametrization, unlike the one with skewness parameter, is continuous in the sense that convergence of parameters holds iff the processes converge.
\end{itemize}

Suppose now that we have a consistent estimator $\theta_n$ of the true parameter~$\theta$. Feasible estimators for supremum, local time and occupation time measure are now obtained via the plug-in approach. In particular, we have 
\begin{align*}
\widetilde T_n^\mean = M_n+\int_0^\infty \left(1-H_n^{\theta_n}(x)\right)\D x, \quad \widetilde T_n^\med= M_n+(H_n^{\theta_n})^{-1}(1/2),
\end{align*} 
where $H_n^{\theta_n}(x)= \prod_{j=0}^{n-1}F_{\theta_n}(n^{1/\a_n}(x+\Delta^n_j),n^{1/\a_n}(\Delta_{j}^n-\Delta_{j+1}^n))$, and 
\begin{align*}
\widetilde L_t(x) &= \frac{1}{\sigma_n \sqrt n} \sum_{i=1}^{\lfloor nt\rfloor} g\left(\frac{\sqrt n}{\sigma_n}(x-X_{\frac{i-1}{n}}), \frac{\sqrt n}{\sigma_n}\Delta_i^n X\right)+O_\p(n^{-1/2}), \\
\widetilde O_t(x) &= \frac 1n \sum_{i=1}^{\lfloor nt \rfloor } 
G \left(\frac{\sqrt n}{\sigma_n} (x-X_{\frac{i-1}{n}}), \frac{\sqrt n}{\sigma_n}   \Delta_i^n X \right)
+O_\p(n^{-1}).
\end{align*}
The construction of estimators $\theta_n$ of the unknown parameter $\theta$ for models (i) and (ii) is a well understood problem in the statistical literature. In particular, in class (i) the maximum likelihood estimator of $\sigma$ is given by
\[
\sigma_n^2 = \sum_{i=1}^n (\Delta_i^n X)^2 
\] 
and it holds that $\sqrt{n}(\sigma_n^2 - \sigma^2) \convd \mathcal N(0, 2 \sigma^4)$. Numerous theoretical results on parametric estimation of model (ii) can be found in e.g. \cite{masuda2015parametric}. Since the maximum likelihood estimator of $\theta$ is not explicit, we rather propose to use the following statistics:
\begin{align*}
\alpha_n&= \frac{q\log(2)}{\log\left(\sum_{i=2}^n |X_{i/n} - X_{(i-2)/n}|^q\right) - \log\left(\sum_{i=1}^n |X_{i/n} - X_{(i-1)/n}|^q\right)}, \\
\rho_n&= \frac{1}{n} \sum_{i=1}^n 1_{\{\Delta_i^n X>0\}}, \qquad \lambda_n = \frac{1}{n} \sum_{i=1}^n \log(n^{1/\a_n} |\Delta_i^n X|),
\end{align*}
where $q\in (-1/2,0)$. Additionally, we need to ensure that our parameters are legal, and in particular $\a_n$, when larger than $1$, is truncated at $(\rho_n\vee(1-\rho_n))^{-1}$.
Due to self-similarity of $X$ and the law of large numbers we have that 
$$\frac{\sum_{i=2}^n |X_{i/n} - X_{(i-2)/n}|^q}{
\sum_{i=1}^n |X_{i/n} - X_{(i-1)/n}|^q} \cip 2^{q/\a},$$ 
which gives the idea behind the construction of $\a_n$. Indeed, all estimators are weakly consistent and since $\e[|X_1|^{2q}]<\infty$
for $q\in (-1/2,0)$  we easily conclude that   
\[
\a_n - \a = O_{\mathbb P}(n^{-1/2}), \qquad \rho_n - \rho = O_{\mathbb P}(n^{-1/2}), \qquad \lambda_n - \lambda = O_{\mathbb P}(n^{-1/2} \log(n)). 
\]
The proposed estimators are not efficient, but they suffice for our purposes, see Proposition \ref{prop:cip0_max} below.

It turns out that the limit theory presented in Theorem~\ref{thm:max} and Theorem~\ref{thm:loc} continues to hold
under a rather  weak assumption on a consistent estimator $\theta_n$ of $\theta$; in particular, this assumption is satisfied by estimators we proposed above. In other words, the difference between the modified and original estimators is negligible in the right sense.

\begin{proposition}\label{prop:cip0_max}
Consider parametric class (i) with $\theta=\sigma$ and $\sigma_n \cip \sigma$ or (ii) with $\theta=(\alpha,\rho,\lambda)$ and $\theta_n\cip \theta$, 
$(\a_n-\a) \log n \cip 0$. Then 
\begin{align*}
&n^{1/\a}(\overline T_n^\mean-\widetilde T_n^\mean)\cip 0,\qquad \text{for } \a\in(1,2],\\
&n^{1/\a}(\overline T_n^\med-\widetilde T_n^\med)\cip 0.
\end{align*}
Moreover, the limit distributions in~\eqref{eq:T1} and~\eqref{eq:T2} are continuous in~$\theta$.
\end{proposition}
This shows that the estimators $\widetilde T_n^\mean$ and $\widetilde T_n^\med$ are asymptotically efficient in the sense that they are asymptotically equivalent to the respective optimal estimators relying on the knowledge of true parameters.
In class (ii) the true $\a$ is not known, but in view of Proposition~\ref{prop:cip0_max} the assumption $(\a_n-\a)\log n\cip 0$ guarantees that
\begin{align*} 
n^{1/\a_n}(\overline X_1-\widetilde T_n^\mean)  &\stab V -\int_0^\infty (1-H(x))\D x,\qquad \text{ when }\a\in(1,2],\\
n^{1/\a_n}(\overline X_1-\widetilde T_n^\med)&\stab V-H^{-1}(1/2).
\end{align*}
Furthermore, the limit distributions are well approximated by their analogues corresponding to parameter~$\theta_n$, and so we may construct asymptotic confidence intervals for the estimators~$\widetilde T_n^\mean$
and~$\widetilde T_n^\med$.

With respect to local/occupation time we have the following result. 
\begin{proposition}\label{prop:cip0}
Consider class (i) and assume that
\begin{equation}\label{eq:rate}
n^{1/4}(\sigma-\sigma_n)\cip 0.
\end{equation} 
Then for any $x\in\R,T>0$ it holds that
\begin{align*}
&n^{1/4}\sup_{t\leq T}\left|\widehat L_t(x)-\widetilde L_t(x)\right|\cip 0, &n^{3/4}\sup_{t\leq T}\left|\widehat O_t(x)-\widetilde O_t(x)\right|\cip 0.
\end{align*}
\end{proposition}
This again shows
that the estimators $\widetilde L_t(x)$ and $\widetilde O_t(x)$ are asymptotically efficient, and provides the respective asymptotic confidence bounds. Condition \eqref{eq:rate} is quite expected in the case of local times since $n^{1/4}$ is the corresponding rate of convergence in
\eqref{eq:loc_thm}, but it is surprising that this condition is also sufficient to conclude the asymptotic efficiency of $\widetilde O_t(x)$. Roughly speaking, the reason for condition \eqref{eq:rate} to be sufficient in the latter case is that partial derivatives of $G$ correspond to the local time asymptotics thus changing the convergence rate from $n^{3/4}$ to $n^{1/4}$. We refer to~\S\ref{unknownpar}  for more details.

\subsection{Truncation of products in supremum estimators}\label{sec:simplified}
Here we return to the assumption that the law of $X$ is known.
Consider supremum estimators defined in~\S\ref{sec:opt_estimators_sup} in terms of the conditional distribution function $H_n(x)$.
When the number $n$ of observations is large, it may be desirable to reduce the number of terms in the product defining $H_n(x)$, in order to avoid numerical issues and to speed-up the calculations. 
This is especially true when $X$ is not a linear Brownian motion and so the function $F$ is not explicit.

Intuitively, we may want to keep the terms which are formed from the observations closest to the maximum.
Thus, we  let $H(x;k)$ for $k\in\mathbb N_+$ be the analogue of $H_n(x)$, but such that the product has at most $2k$ terms and, in particular, the indices $j$ are chosen such that $0\vee (I_n-k)\leq j\leq (I_n+k-1)\wedge (n-1)$ with $I_n$ being the index of the maximal observation.
%That is, we use $k$ terms formed from the observations immediately on the left of the maximum, and another $k$ terms formed from the observations immediately on the right, when possible.
Define $\overline T^\mean_{n,k}$ and $\overline T^\med_{n,k}$ as before but using $H_n(x;k)$ instead of $H_n$.
%It is noted that $H_{n,k}(x)$ is a cdf of a certain distribution on $\R_+$, and $T_{n,k},\widehat T_{n,k}$
%are the corresponding mean and median bth shifted by the maximum~$m$.

Letting $I\in\mathbb Z$ be the unique number satisfying $\xi_{I+U}=V$ (it achieves the minimum $V$ in~\eqref{Maxconv}), we define
\begin{align*}
H(x;k)&:=\prod_{I-k\leq j\leq I+k-1}F\left(x+\xi_{j+U}-V,\xi_{j+U}-\xi_{j+1+U}\right),\quad x\geq 0.\end{align*}
We now have the limit result analogous to Theorem~\ref{thm:max}:
\begin{corollary}\label{cor:truncation}
For any $\a\in(0,2]$ it holds that
\begin{align}
n^{1/\a}(\overline X_1-\overline T^\mean_{n,k})  &\stab V -\int_0^\infty (1-H(x;k))\D x,\\
n^{1/\a}(\overline X_1-\overline T^\med_{n,k})&\stab V-H^{-1}(1/2;k).
\end{align}
\end{corollary} 
It turns out that we do not need to exclude $\a\in(0,1]$ in the case of modified conditional mean estimator, because the number of terms is kept finite.

\subsection{Comments on the  general case in supremum estimation}
It is likely that Theorem~\ref{thm:max} can be generalized to an arbitrary L\'evy process satisfying
the following weak regularity condition:
\[ (a_u X_{t/u})_{t\geq 0} \convd (\widehat X_{t})_{t\geq 0}\qquad\text{ as }u\to\infty,\]
for some positive function~$a_u$ and necessarily self-similar \levy process $\widehat X$.
Importantly, the general versions of~\eqref{eq:zooming} and~\eqref{Maxconv} are proven in~\cite{ivanovs_zooming}; here the limiting objects correspond to~$\widehat X$.

There are, however, two very serious difficulties. Firstly, joint convergence does not necessarily imply convergence of the conditional distributions. Thus, one needs to use the underlying structure to show that 
\[F_{1/n}(x/a_n,y/a_n)=\p(a_n\overline X_{1/n}\leq x|a_n X_{1/n}=y)\to \widehat F(x,y).\]
Secondly, the proof of uniform negligibility of truncation in~\S\ref{sec:uniform_neg} crucially depends on $X$ being self-similar. This part may be notoriously hard for a general L\'evy process~$X$.

\section{Numerical illustration of the limit laws}\label{sec:numerics}
In this section we perform some numerical experiments in order to illustrate the limit laws in Theorem~\ref{thm:max} 
and Theorem~\ref{thm:loc}. For simplicity we take $X$ to be a standard Brownian motion and, additionally, a one-sided stable process in supremum estimation which is motivated by the semi-explicit formula for the function $F$ in Proposition~\ref{prop:F}.
All the densities are obtained from $10,000$ independent samples using standard kernel estimates. The number of samples is reduced to $1,000$ in the case of the one-sided stable process.

\subsection{Supremum estimation for Brownian motion}
Consider a standard Brownian motion $X$ and the limiting random variable $V$ in~\eqref{Maxconv}, as well as
$V_\text{\rm mean}:=V-\int_0^\infty (1-H(x))\D x$ and $V_\text{\rm med}:=V-H^{-1}(1/2)$ in~\eqref{eq:T1} and~\eqref{eq:T2}, respectively.
Recall that all of these quantities are explicit, see also Corollary~\ref{cor:BM}, but they all depend on infinitely many observations $\xi_{j+U},j\in\mathbb Z$ of the two-sided 3-dimensional Bessel process~$\xi$.
We approximate these quantities by setting $\xi_{j+U}=\infty$ for $j<-50$ or $j\geq 50$, which effectively amounts to considering 100 epochs centered around~0; choosing twice as many epochs had negligible effect on the results below. 
The resulting densities are depicted in Figure~\ref{fig:max}.
In Table~\ref{table:bm} we report the Root Mean Squared Error (RMSE), Mean Absolute Error (MAE), and the narrowest $95\%$-confidence interval length
for each of the limiting distributions. It is noted that, indeed, $V_\text{\rm mean}$ has the smallest RMSE and $V_\text{\rm med}$ has the smallest MAE, and the respective distribution are very similar.
\begin{figure}[h!]
	\includegraphics[width=0.7\textwidth]{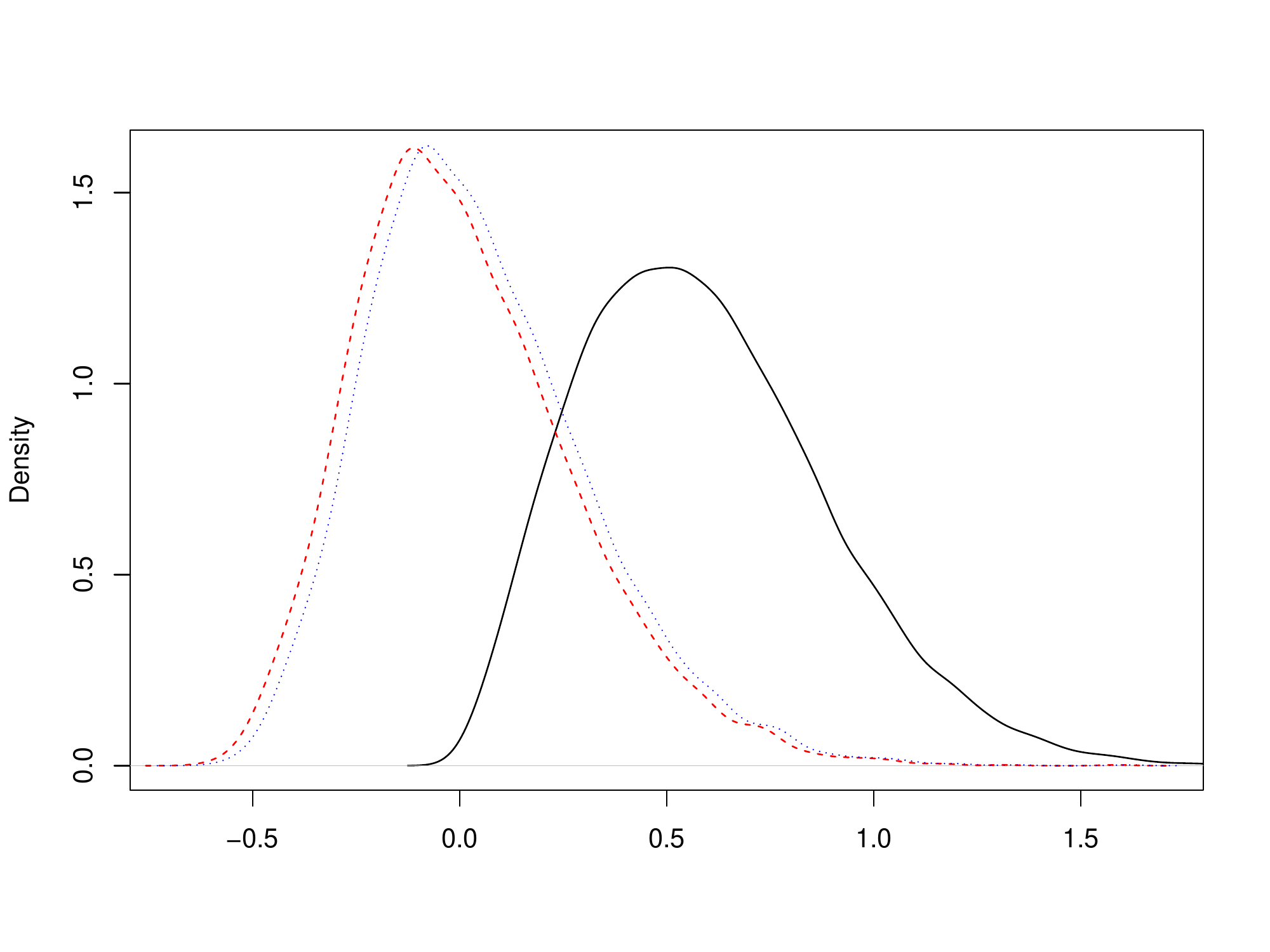}
	\caption{Simulated densities of $V$ (solid black), $V_\text{\rm mean}$ (dashed red) and $V_\text{\rm med}$ (dotted blue) in the Brownian case}
	\label{fig:max}
\end{figure}
\begin{table}
	\caption{Some statistics in the Brownian case}\label{table:bm}
	\begin{tabular}{c|c|c|c|c|c}
		&$V$ & $V_\text{\rm mean}$ & $V_\text{\rm med}$ & $V_\text{\rm shift}$ & $V^1_\text{\rm mean}$\\
		\hline
		RMSE & 0.66 & 0.26 & 0.27 &0.30&0.29\\
		MAE & 0.59 & 0.21 & 0.21 &0.24&0.22\\
		$95\%$ conf.\ int.\ length &1.14&1.03&1.03&1.14&1.06
	\end{tabular}
\end{table}

Observe that the main problem of the standard estimator $M_n$ is that it is downward biased and so $V$ is not centered.
This, however, can be easily remedied since according to~\cite{asmussen_glynn_pitman1995} \[\e V=-\zeta\left(\frac{1}{2}\right)\frac{1}{\sqrt{2\pi}}\approx 0.5826,\]
where $\zeta$ is the Riemann zeta function. In other words, we may consider an asymptotically centered estimator $M_n+\frac{1}{\sqrt n}\e V$, which leads to $V_\text{\rm shift}:=V-\e V$. Finally, we also consider the truncated conditional mean estimator $\overline T^\text{\rm mean}_{n,1}$ based on $H(x;1)$, which is a product of two terms and thus only moderately more complicated to evaluate as compared to $M_n$, see~\S\ref{sec:simplified}. The respective limit is denoted by $V^1_\text{\rm mean}$. Relative comparison of the latter two together with $V_\text{\rm mean}$ is provided in Figure~\ref{fig:max2}, see also Table~\ref{table:bm}.
\begin{figure}[h!]
	\includegraphics[width=0.7\textwidth]{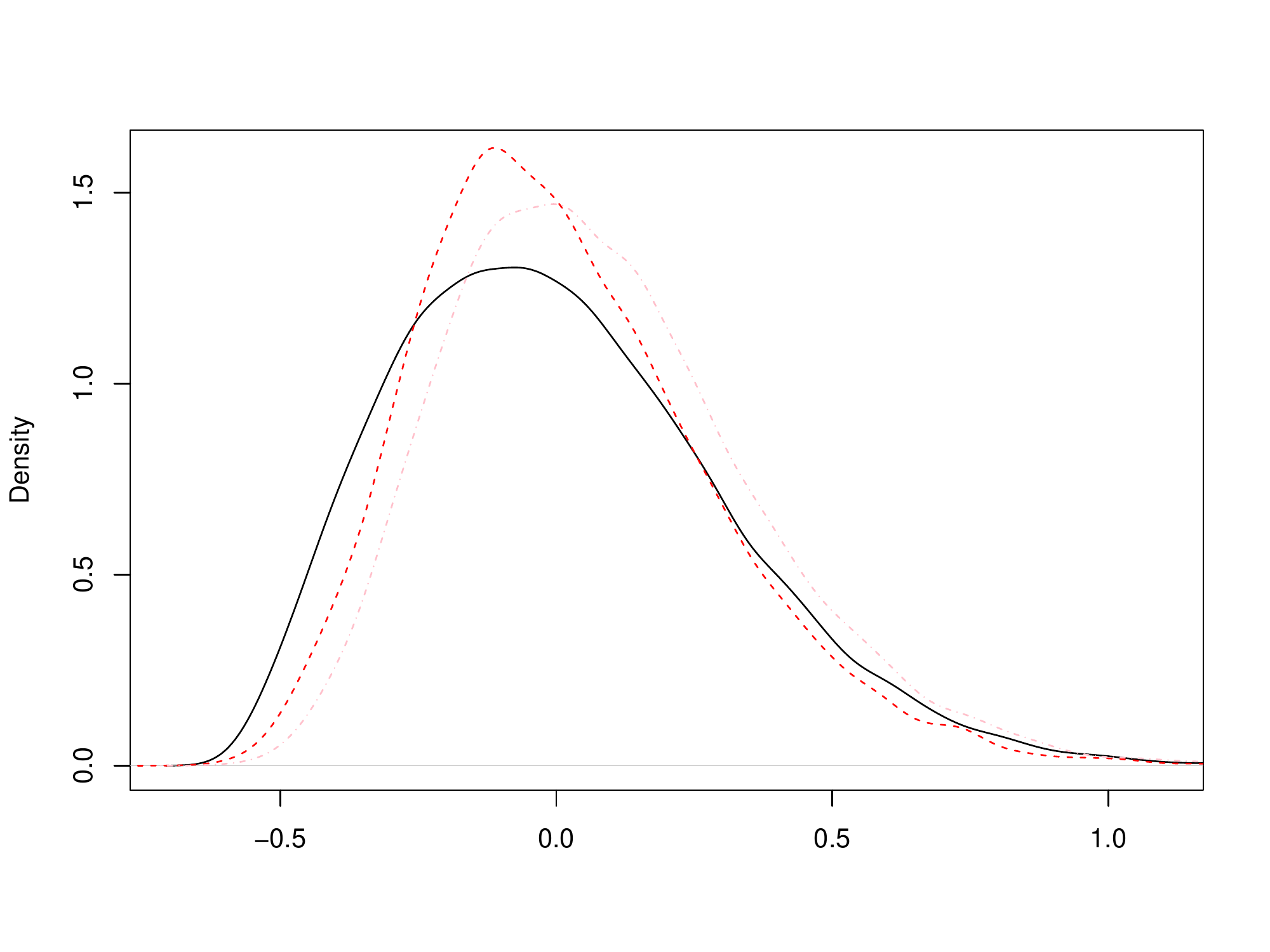}
	\caption{Simulated densities of $V_\text{\rm shift}$ (solid black), $V_\text{\rm mean}$ (dashed red) and $V^1_\text{\rm mean}$ (dot-dashed pink) in the Brownian case}
	\label{fig:max2}
\end{figure}

In conclusion, the conditional mean and conditional median estimators are very similar to each other and considerably better than the standard estimator~$M_n$ in terms of RMSE and MAE. Nevertheless, the other simple estimators discussed above are only slightly worse than the optimal ones.

\subsection{Supremum for one-sided stable process}
Here we consider a strictly stable L\'evy process with $\a=1.8$, standard scale and only negative jumps present, i.e., the skewness parameter is $\beta=-1$.
 Note that the results in the opposite case $\beta=1$ must be similar according to Proposition~\ref{prop:symmetry}. 
The conditional distribution function $F$ is numerically evaluated using the expressions in Proposition~\ref{prop:F}, see Figure~\ref{fig:F}.
\begin{figure}[h!]
	\centering
	\begin{subfigure}[b]{0.47\textwidth}
		\includegraphics[width=\textwidth]{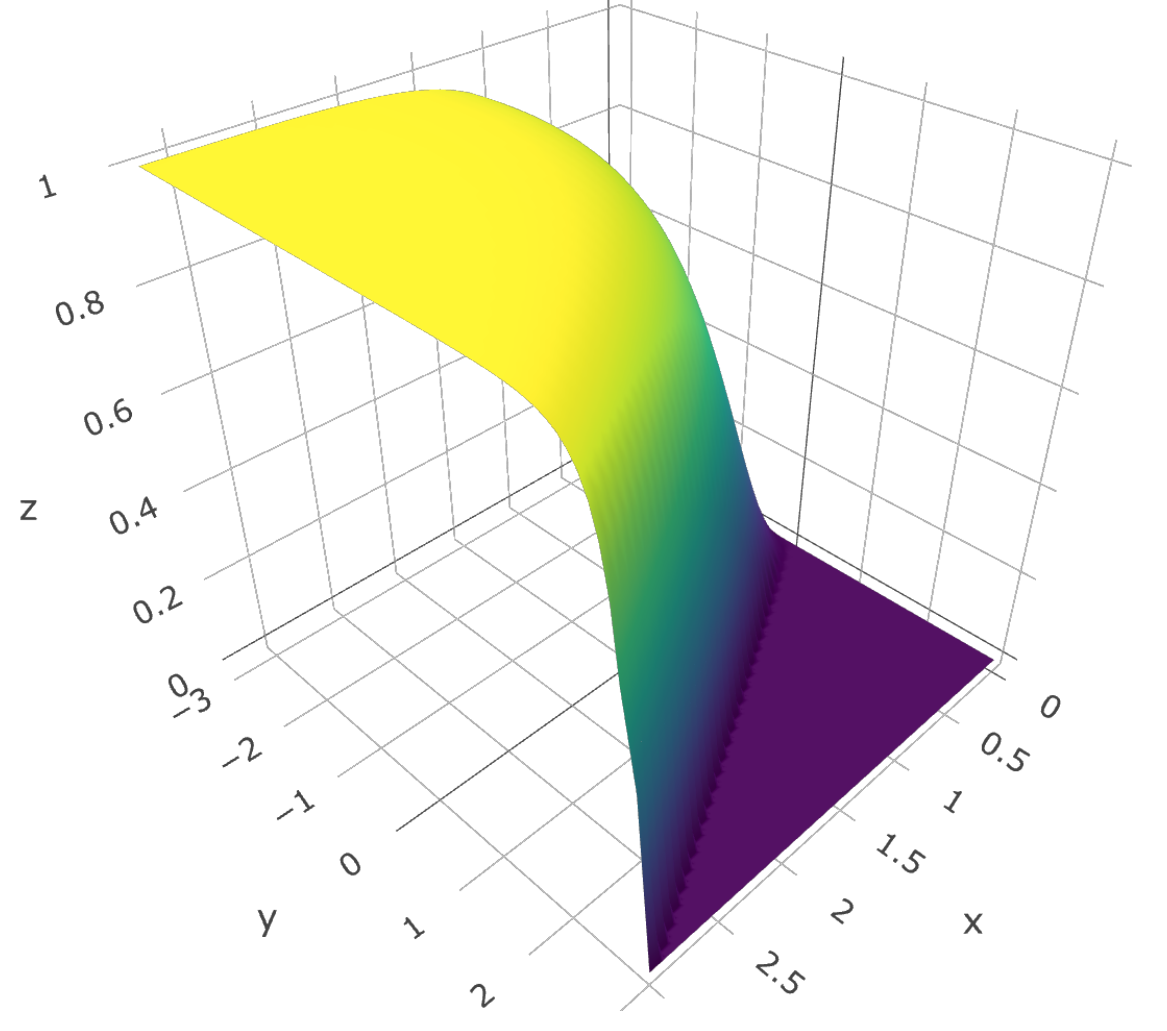}
		\caption{The conditional distribution function $F(x,y)$}
		\label{fig:F}
	\end{subfigure}
	\quad %add desired spacing between images, e. g. ~, \quad, \qquad, \hfill etc. 
	\begin{subfigure}[b]{0.47\textwidth}
		\includegraphics[width=\textwidth]{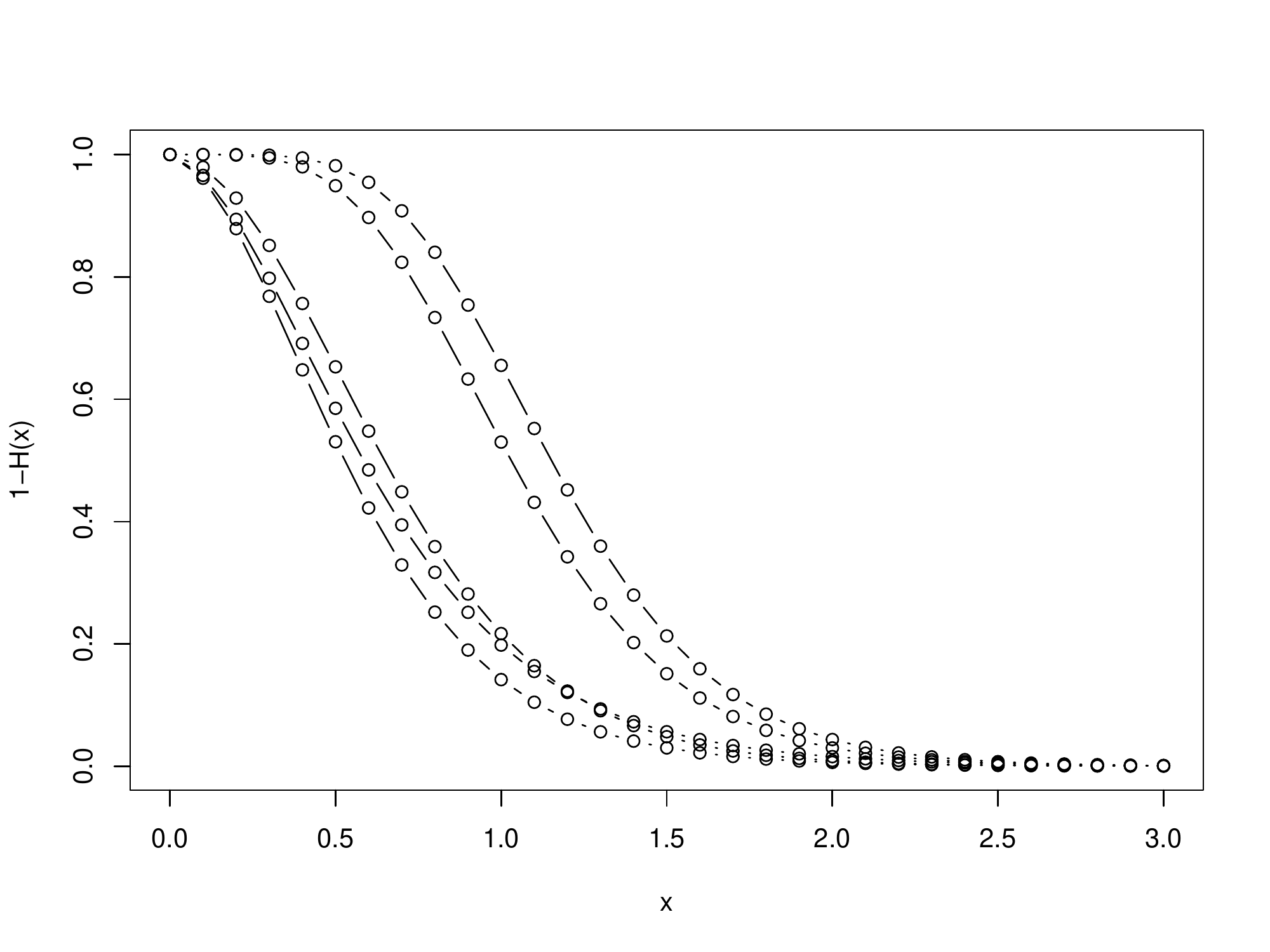}
		\caption{Samples of $1-H(x)$ approximations}
				\label{fig:H}
	\end{subfigure}
\caption{The stable case}
\end{figure}

In this case we perform a number of approximations. 
Firstly, simulation of $\xi$ is not obvious (unlike the Brownian case) and so we approximate the limiting object $(V,H(x))$ by $(n^{1/\a}(\overline X_1-M_n),H_n(x n^{-1/\a}))$ with $n=300$, see~\eqref{eq:Hconv}. Instead of scaling $\overline X_1,M_n,\Delta_i$ with $n^{1/\a}$ we perform the simulation of the process $X$ on the interval $[0,n]$, 
which is allowed by self-similarity of~$X$. 
Furthermore, $X$ is simulated on the grid with step-size $1/m$ for $m=300$, which yields an approximation of $\overline X_n$ further corrected by the easily computable asymptotic mean error $m^{-1/\a}\e V$, see~\cite{asmussen_ivanovs_twosided}.
Next, we take (at most) $30$ terms in the product defining $H_n$ based on the observations closest to the maximum, that is we replace it by $H_n(\cdot;15)$ defined in~\S\ref{sec:simplified}. Finally, $\int_0^\infty(1-H(x))\D x$ is approximated using the trapezoidal rule with step size $0.1$ and truncation at~$x=3$, see Figure~\ref{fig:H}; the same approximation is used in calculation of the inverse.

The results are presented in Figure~\ref{fig:results} and Table~\ref{table:stable}. They are quite similar to the results in the Brownian case. 
\begin{table}
	\caption{Some statistics in the stable case}\label{table:stable}
\begin{tabular}{c|c|c|c|c|c}
	&$V$ & $V_{mean}$ & $V_{med}$ & $V_\text{\rm shift}$\\
	\hline
	RMSE &0.87&0.41 & 0.42 & 0.43\\
	MAE & 0.75 & 0.32 & 0.32 & 0.34\\
	$95\%$ conf.\ int.\ length &1.45&1.42&1.42 &1.45
\end{tabular}
\end{table}
\begin{figure}[h!]
	\includegraphics[width=0.7\textwidth]{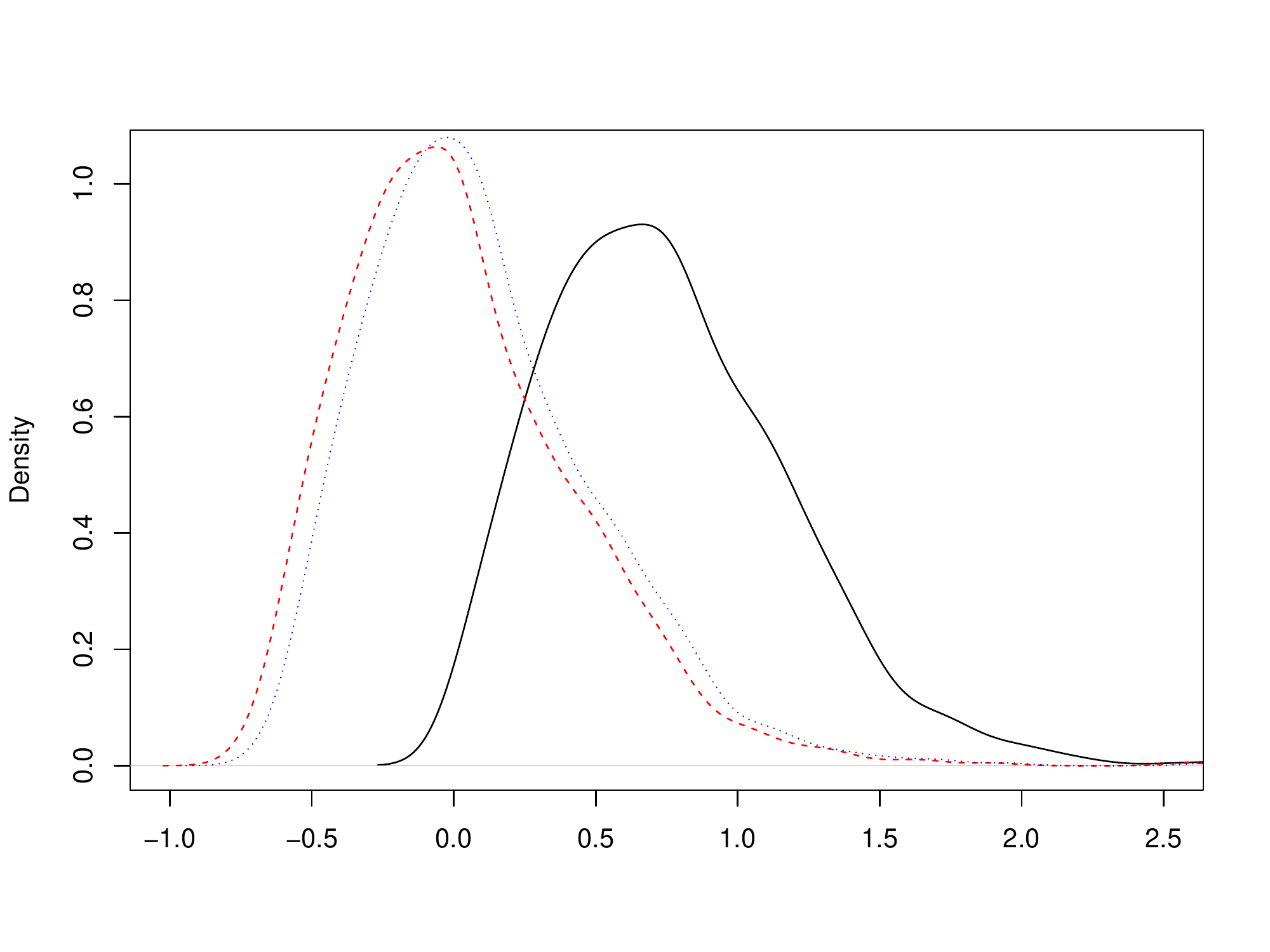}
	\caption{Simulated densities of $V$ (solid black), $V_\text{\rm mean}$ (dashed red) and $V_\text{\rm med}$ (dotted blue) in the stable case}
	\label{fig:results}
\end{figure}

\subsection{Local time and occupation time for Brownian motion}
Let again $X$ be a standard Brownian motion and choose $x=0,t=1$.
We use $\widehat L_1(0)$ with $n=10,000$ as a substitute for the true $L_1(0)$, which then allows to sample (approximately) from the limit distribution in~\eqref{eq:loc_thm}. Next, we use the same sample path to construct $\widehat L_1(0)$ with $n=100$, which  allows to sample from the pre-limit expression in~\eqref{eq:loc_thm}. Finally, we also take a standard estimator $\frac{1}{2\sqrt n}\#\{i\in[0:n-1]: |X_{i/n}|<\frac{1}{\sqrt n}\}$ for $n=100$.
The respective densities are depicted in Figure~\ref{fig:loc_time}. The ratio of variances for $n=100$ is $1:1.64$, which can be compared to $1:1.35$ for the more advanced estimator mentioned in Remark~\ref{rem:loc_alt} (here we use the exact expressions of the limits).
\begin{figure}[h!]
	\includegraphics[width=0.7\textwidth]{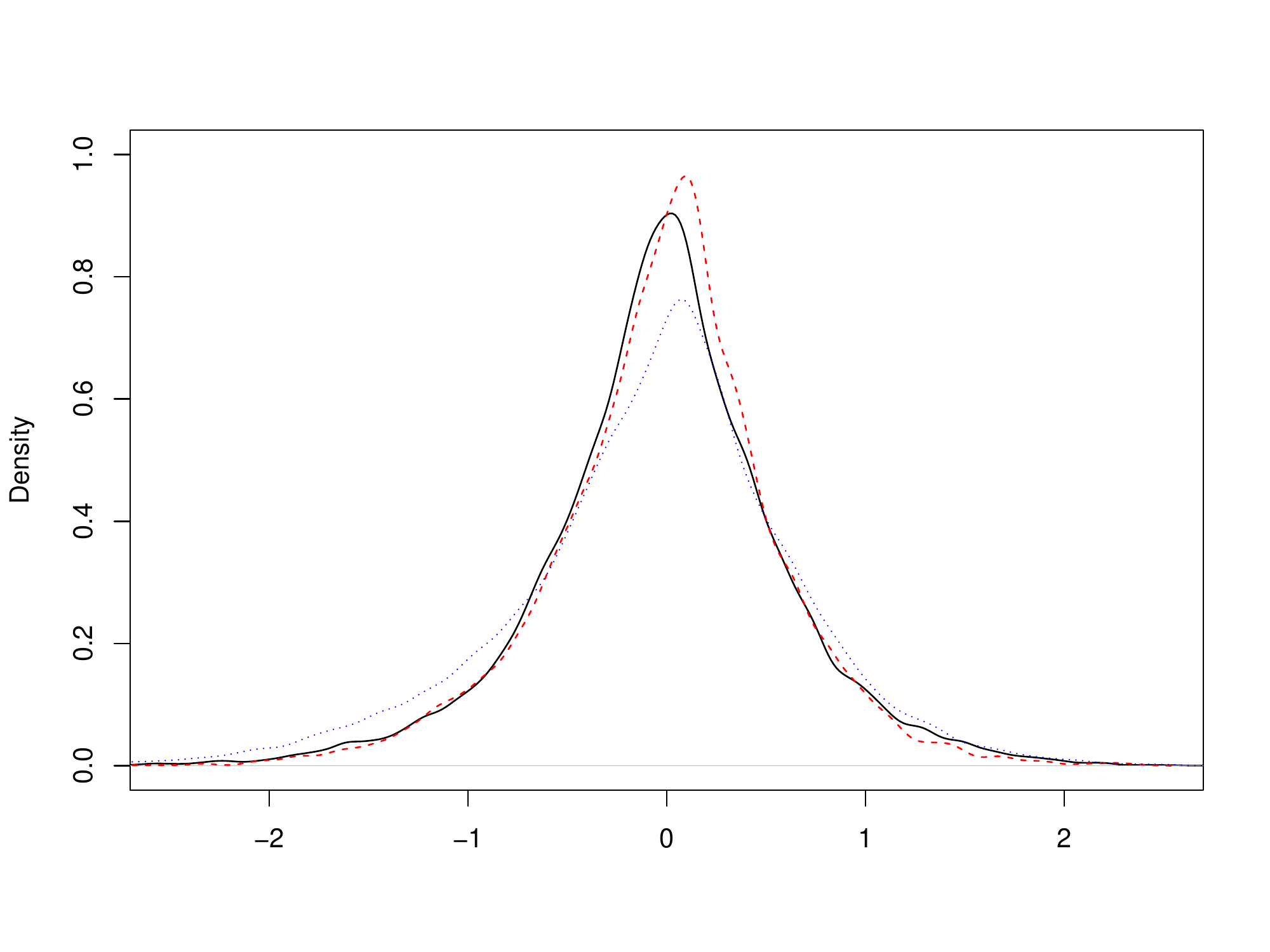}
	\caption{Local time: the densities of the limit (solid black) and pre-limit (dashed red) in~\eqref{eq:loc_thm} for $n=100$, as well as pre-limit for the standard estimator (dotted blue)}
	\label{fig:loc_time}
\end{figure}

We perform a similar procedure for the occupation time in $(0,\infty)$. Here the standard estimator is $\frac{1}{n}\#\{i\in[0:n-1]:X_{i/n}\geq 0\}$.
The respective densities are given in Figure~\ref{fig:occ_time}, and we see a very substantial improvement. The ratio of variances for $n=100$ is $1:2.64$.
\begin{figure}[h!]
	\includegraphics[width=0.7\textwidth]{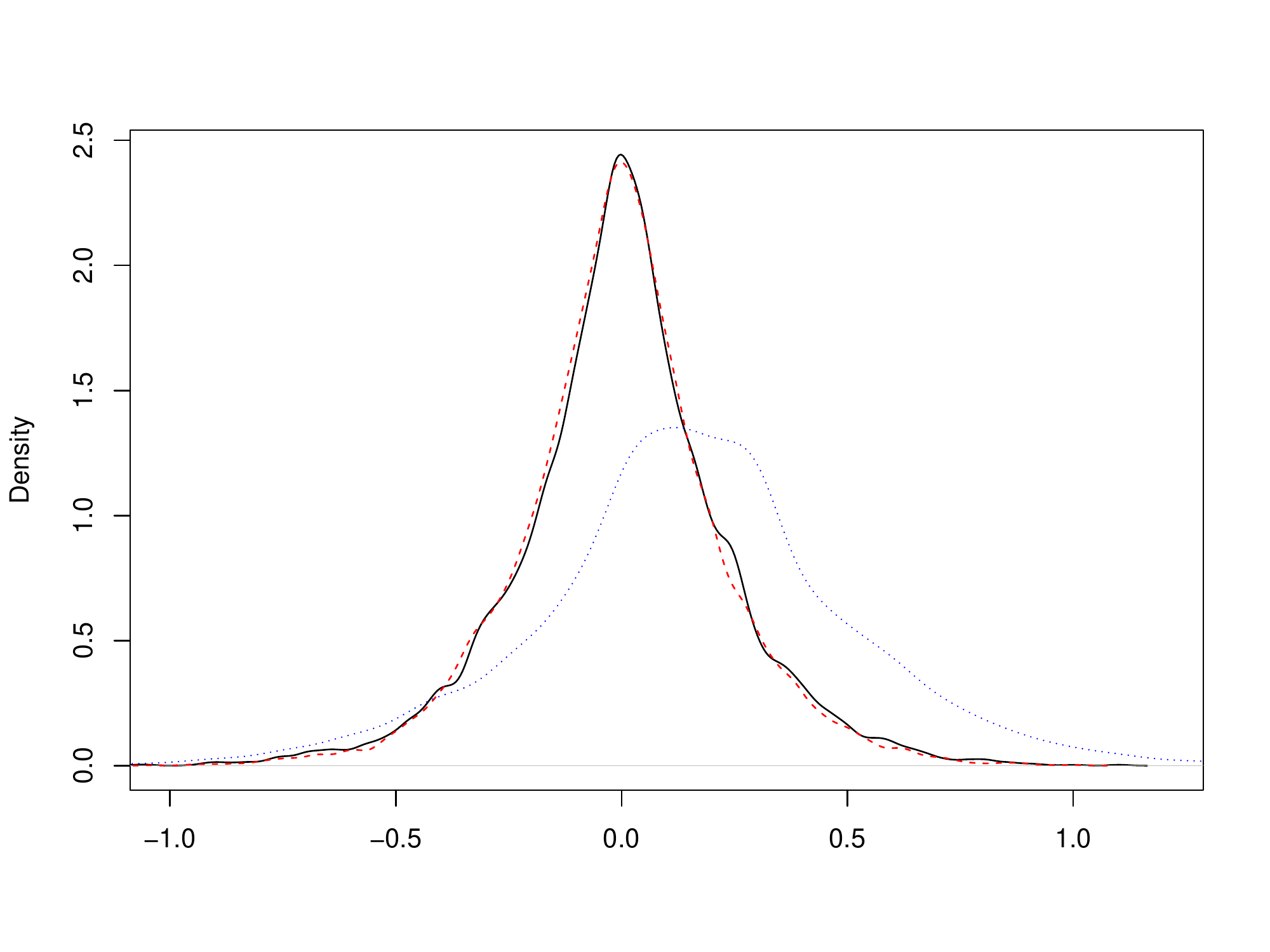}
	\caption{Occupation time: the densities of the limit (solid black) and pre-limit (dashed red) in~\eqref{clton} for $n=100$, as well as pre-limit for the standard estimator (dotted blue)}
	\label{fig:occ_time}
\end{figure}

\section*{Acknowledgements}
We would like to express our gratitude to Johan Segers for his comments concerning pre-estimation of model parameters.

\appendix
\section{Proofs for supremum estimation}\label{sec:proofs_sup}
In the following all positive constants will be denoted by $c$ although the may change from line to line.

\subsection{Duality}
In this section we establish a duality result for a general L\'evy process~$X$. 
Even though it is not needed for the proofs, we present this duality, because it explains certain structure in the main results.
To this end, consider the process $X'_t=-X_t$ and the associated quantities $\overline{X'}_1, M'_n,H'_n(x),F'_t(x,y)$, see~\S\ref{sec:opt_estimators_sup}.
\begin{proposition}\label{prop:symmetry}
	Let $X$ be an arbitrary L\'evy process. Then
		\[\left(\overline{X'}_1-M'_n,(H_n'(x))_{x\geq 0}\right)\,\stackrel{d}{=}\,\left(\overline X_1-M_n,(H_n(x))_{x\geq 0}\right).\]
	Furthermore, $F'_t(x,y)=F_t(x-y,-y)$.
\end{proposition}
\begin{proof}
We take $(X''_t)_{t\in[0,1]}:=(X_{(1-t)-}-X_1)_{t\in[0,1]}$, which has the law of $(-X_t)_{t\in[0,1]}$ (standard time-reversal). Then
$\overline{X''}_1-M''_n=\overline X_1-X_1-(M_n-X_1)=\overline X_1-M_n$, because $X$ does not jump at $j/n$ almost surely. 
Letting $x_j$ be the observation of $X_{j/n}$ we find that
\begin{align*}
&H_n(x)=\p(\overline{X''}_1-M''_n\leq x| X_{k/n}-X_1=x_k-x_n\, \forall k\in[1:n-1], X_1=x_n)\\
&=\p(\overline{X''}_1-M''_n\leq x| X''_{(n-k)/n}=x''_{n-k}\, \forall k\in[1:n-1], X_1''=x''_n)=H'_n(x).
\end{align*}
Finally, 
\begin{align*}
F'(x,y)&=\p(\overline{X''}_1\leq x|X''_1=y)=\p(\overline X_1-X_1\leq x|X_1=-y)\\
&=\p(\overline X_1\leq x-y|X_1=-y)=F(x-y,-y)
\end{align*}
and the same reasoning works for $F'_t(x,y)$ when time-reverting at~$t$.
\end{proof}
In view of Proposition~\ref{prop:symmetry}, the errors $\overline X_1-\overline T_n^\mean$ and $\overline X_1-\overline T_n^\med$
have the same distribution as the respective errors for the process~$-X$. Thus, the corresponding limit results must stay the same when the skewness parameter $\beta$ is flipped to the opposite. In the proofs we may safely assume that $\beta\geq 0$, say.  

\subsection{On the function $F$ in the stable case}
Before starting the proof of the main result we establish some basic properties of the conditional probability $F(x,y)$ in the case of a strictly $\a$-stable process when it is not explicit.
Throughout this subsection we assume that $X$ is a strictly $\a$-stable process with skewness parameter $\beta\in[-1,1]$. Note that the boundary values $\beta=-1$ and $\beta=1$ correspond to spectrally negative and spectrally positive processes, respectively; in both cases we must have $\a\in(1,2)$, because we have excluded monotone processes.

It is well known~\cite[p.\ 88]{sato} that $X_t$ has a continuous strictly positive bounded density, call it $f_t(x)$.
Moreover, by self-similarity 
\begin{equation}\label{eq:f_t}f_t(x)=t^{-1/\a}f(t^{-1/\a}x)\qquad\text{ with }f=f_1.\end{equation}
Furthermore, $f(x)\sim cx^{-\a-1} $ as $x\to\infty$ when $\beta\neq -1$, and otherwise it decays faster than an exponential function~\cite[Eq. (14.37)]{sato}.
 %, which is unimodal~\cite[p.\ 404]{sato}.

Let us define the first passage times $\tau^\pm_x=\inf\{t\geq 0:\pm X_t>\pm x\}$ above and below a given level $x$.
\begin{lemma}\label{lem:F}
The function $F(x,y)$ is jointly continuous. Moreover, $F(x,y)=0$ for $x\leq y_+$, and otherwise
\begin{align}\overline F(x,y)f(y):&=(1-F(x,y))f(y)\nonumber\\
&=\e\left[(1-\tau^+_x)^{-1/\a}f((1-\tau^+_x)^{-1/\a}(y-X_{\tau^+_x}));\tau^+_x<1\right]\label{eq:F1}\\
&=\e\left[(1-\tau^-_{y-x})^{-1/\a}f((1-\tau^-_{y-x})^{-1/\a}(y-X_{\tau^-_{y-x}}));\tau^-_{y-x}<1\right],\label{eq:F2}
\end{align}
where $\e[Y;A]= \e[Y1_A]$.
%Furthermore, the function
%\[\int_0^\infty \left(1-F(z+x,y)\right)\D z\]
%is finite and jointly continuous in $(x,y)$ on the same domain.
\end{lemma}
\begin{proof}
Assume for the moment that $x> y_+$. By time reversal (or from Proposition~\ref{prop:symmetry}) we get 
\begin{align} \label{eq:timerev}
\p(\overline X_1>x,X_1\in\D y)=\p(\underline X_1<y-x,X_1\in\D y).
\end{align}
Using the strong Markov property we find that
\[\iint_{t\in(0,1),z\geq x}\p(\tau^+_x\in\D t,X_{\tau^+_x}\in\D z)f_{1-t}(y-z)\]
is a version of the density of the measure on the left of \eqref{eq:timerev}. This expression coincides with~\eqref{eq:F1} according to~\eqref{eq:f_t}.
Similarly, \eqref{eq:F2} is a version of the density of the measure on the right of \eqref{eq:timerev}, and hence both expressions coincide for almost all~$y$.

Next, we show that the expressions in~\eqref{eq:F1} and~\eqref{eq:F2} are jointly continuous on $x>y_+$, and thus must coincide on this domain.
We do this for the first expression only, since the other can be treated in the same way.
By the basic properties of L{\'e}vy processes~\cite{bertoin} we see that $\tau^+_x\neq 1$ and $(\tau^+_{x},X_{\tau^+_{x}})$  is continuous on an event of probability~1. Hence we only need to show that the dominated convergence theorem applies. 
Choose an arbitrary sequence $(x',y')$ converging to $(x,y)$ with $x>y_+$. Now $X_{\tau^+_{x'}}-y'>x'-y'>\epsilon$ for some $\epsilon>0$ (further down in the sequence). 
Note that $f(-x)\leq cx^{-\a-1}$ for some $c>0$ and all $x>0$; in the spectrally positive case the decay is even faster. Hence the term under the expectation is bounded by
$c(1-\tau^+_{x'})\epsilon^{-\a-1}$ and we are done.

It is left to show that either one of~\eqref{eq:F1} and~\eqref{eq:F2} converges to $f(y')$ as $x\to x',y\to y'$ with $x<y_+$ and $x'=y'_+$ (the boundary of the domain); this would imply $F(x,y)\to 0$.
In the case $y'<0$ use~\eqref{eq:F1} and the above reasoning, while for $x'>0$ use~\eqref{eq:F2}.
It is left to analyze the case of $x'=y'=0$. Note that~\eqref{eq:F1} is lower bounded by the same expression with the indicator replaced by the indicator of $\tau_x^+<1/2$. But now the dominated convergence theorem applies and yields the limit $f(0)$. The upper bound is $f(y)$ by construction, and the limit is again~$f(0)$. The proof is thus complete.
\end{proof}

We are now ready to provide some bounds on $\overline F(x,y)$. In the one-sided cases the bounds can be considerably improved, but this is not needed in this work and so we prefer a simpler statement.
\begin{lemma}\label{lem:bounds}
There exists a constant $c>0$ such that for all $x\geq y_+$:
\begin{align*}
&\overline F(x,y)\leq c x^{-\a}(x-y)^{-\a-1}(|y|\vee 1)^{\a+1}, &\beta\in(-1,1),\\
&\overline F(x,y)\leq c \exp(-(x-y_+)), &\beta=\pm 1.
\end{align*}
\end{lemma}
\begin{proof}
Suppose that $\beta\in(-1,1)$.  We know that $f(x)<c|x|^{-\a-1}$. According to~\eqref{eq:F1} we then have
\[\overline F(x,y)f(y)\leq c \e[(X_{\tau^+_x}-y)^{-\a-1}(1-\tau_x^+);\tau_x^+<1]\leq c(x-y)^{-\a-1}\p(\overline X_1>x).\]
It is left to recall that $\p(\overline X_1>x)\sim c x^{-\a}$ as $x\to \infty$ when $\beta\neq -1$~\cite{bingham,doney_savov}.

Assume that $\beta=-1$. In this case $f(x)\sim ax^b\exp(-ux^v)$ with $a,u>0$ and $v>1$ as $x\to\infty$, see~\cite[Eq. (14.37)]{sato}.
Furthermore, the asymptotics of $\p(\overline X_1>x)$ has a similar form~\cite[Prop.\ 3b]{bingham}.
Observe that $f(Ax+B)<cA^{-1}f(x)$ for  all $A\geq 1,B\geq 0,x\geq 1$.
Hence, from~\eqref{eq:F2} we get the bound
\[\overline F(x,y)\leq c f(x)/f(y),\] which for $y>0$ leads to the claimed bound~$c\exp(-(x-y))$.
For $y\leq 0$ we find from~\eqref{eq:F1} that
\[\overline F(x,y)\leq c (x-y)^{-\a-1}\exp(-x)/f(y),\]
which readily implies the bound $c\exp(-x)$. 
Similar analysis yields the bound in the case $\beta=1$.
\end{proof}
It is noted that we may also derive a bound
\[\overline F(x,y)\leq c x^{-\a-1}(x-y)^{-\a}(|y|\vee 1)^{\a+1}\]
for $\beta\in(-1,1)$ by using~\eqref{eq:F2} instead of~\eqref{eq:F1}. This bound is better when $y>0$ and worse when $y<0$. For our purpose any of these bounds is sufficient.

Finally, we derive a semi-explicit expression of $F(x,y)$ in the one-sided case. This expression is in terms of the density~$f$.
\begin{proposition}\label{prop:F}
	In the spectrally one-sided cases we have for all $x>y_+$:
	\begin{align*}
	&\beta=-1:\\&\quad\overline F(x,y)=\frac{x}{f(y)}\int_0^1(1-t)^{-1/\a}t^{-1/\a-1}f(xt^{-1/\a})f(y-x)(1-t)^{-1/\a})\D t,\\
	&\beta=1:\\&\quad\overline F(x,y)=\frac{x-y}{f(y)}\int_0^1(1-t)^{-1/\a}t^{-1/\a-1}f(x(1-t)^{-1/\a})f((y-x)t^{-1/\a})\D t.
	\end{align*}
\end{proposition}
\begin{proof}
	It is known that $\p(\overline X_1\in \D x)=\a f(x)$ for $x>0$, when $X$ is spectrally negative, see~\cite{Michna}.
	Moreover,
	\[\p(\tau^+_x<t)=\p(\overline X_t>x)=\p(\overline X_1>xt^{-1/\a})\]
	yielding that $\p(\tau_x^+\in\D t)=xt^{-1/\a-1}f(xt^{-1/\a})\D t$. Plugging this into Lemma~\ref{lem:F} yields the result.
	Finally,~\eqref{eq:F2} follows from $\overline F'(x,y)=\overline F(x-y,-y)$, see Proposition~\ref{prop:symmetry}.
\end{proof}

\begin{remark}\label{rem:finiteness}\rm
Note that $\e(\overline X_1|X_1=y)<\infty$ for all $y\in\R$, even in the cases $\a\in(0,1]$ where $\e \overline X_1=\infty$. This follows from Lemma~\ref{lem:bounds} showing that for  fixed $y$ we have a bound
$\overline F(x,y)\leq c x^{-2\a-1}$. Thus conditional moments of order up to $1+2\a$ exist. In the spectrally-positive case we even have $\e(\exp(\lambda \overline X_1)|X_1<b)<\infty$ for any $b<\infty,\lambda>0$ (Lemma~\ref{lem:bounds} gives only $\lambda<1$ though).
\end{remark}

%======================================== PROOF =================

\subsection{Proof of Theorem~\ref{thm:max}}
In the following we frequently use the inequality
\begin{equation}\label{eq:ineq}|\prod_{j\in \mathbb Z} a_j - \prod_{j\in \mathbb Z} b_j |\leq 
\sum_{j\in \mathbb Z} |a_j-b_j| \text{ when }a_j,b_j \in (0,1). \end{equation}

%Weak convergence of $V^\n = n^{-1/\a}(\overline X_1-m)$ is proven in~\cite{asmussen_glynn_pitman1995} for the case of a Brownian motion and in~\cite{ivanovs_zooming} for a general L\'evy process. 

Let $I^\n=\lceil\tau n\rceil$ be the index of the first observation to the right of the supremum time, and put 
\[u_i^\n=\begin{cases}n^{1/\a}(\overline X_1-X_{(i+I^\n)/n}),&i+I^\n\in[0,n]\\ \infty,&\text{otherwise}.\end{cases}\] 
In other words, $u_i^\n$ are the rescaled distances from the supremum to the observations indexed with respect to the time of supremum.
Now we can represent the quantities appearing in~\eqref{eq:Hconv} as follows:
\begin{align*}
V^\n &:=n^{1/\a}(\overline X_1-M_n)=\min_{i\in \mathbb{Z}} u_i^\n,\\
H^\n(x)&:=H_n(xn^{-1/\a})= \prod_{i\in \mathbb{Z}}F(x+u_i^\n-V^\n,u_i^\n-u_{i+1}^\n),
\end{align*}
where by convention $F=1$ if either of $u^\n_i,u_{i+1}^\n$ is infinite.
According to~\cite{ivanovs_zooming} (or~\cite{asmussen_glynn_pitman1995} in the case of Brownian motion) we have the following weak convergence for every $k>0$:
\begin{equation}\label{eq:zooming_conv}\left((u_i^\n)_{|i|\leq k},V^\n\right)\,\stab\,\left((\xi_{i+U})_{|i|\leq k},\min_{i\in\mathbb Z}\xi_{i+U}\right).\end{equation}
Intuitively, this limit can be understood as arising from~\eqref{eq:zooming} together with the fact that $\{n\tau\}$ converges to an independent uniform
on $(0,1)$.
This explains (of course, only intuitively) the form of the result in Theorem~\ref{thm:max}.

\subsubsection{Convergence of the truncated versions}
Let $H_k^\n$ be the same as $H^\n$, but with the product running over $|i|\leq k$:
\[H^\n_k(x)=\prod_{|i|\leq k}F(x+u_i^\n-V^\n,u_i^\n-u_{i+1}^\n),\] where again $F=1$ when the index is out of range.
We also define the analogous object formed from the limiting quantities:
\[H^{(\infty)}_k(x)=\prod_{|i|\leq k}F\left(x+\xi_{j+U}-V,\xi_{j+U}-\xi_{j+1+U}\right).\]
Note that $H^\n_k(x),H^{(\infty)}_k(x)$ are continuous and strictly increasing in $x\geq 0$ which is inherited from~$F(x,y)$.
Furthermore, $H^\n_k(\infty)=H^{(\infty)}_k(\infty)=1$, whereas their value at 0 is not necessarily~0.
In the following the inverse of an increasing function $f$ is defined as usual: $f^{-1}(q)=\inf\{s:f(s)\geq q\}$.
\begin{lemma}\label{lem:conv_trunc}
For any $k\in\mathbb N$ as $n\to\infty$ we have \[(V^\n,H^\n_k(x)_{x\geq 0})\stab (V,H^{(\infty)}_k(x)_{x\geq 0})\] 
with respect to  the uniform topology. Moreover,
\begin{align*}
\left(V^\n,\int_0^\infty (1-H^\n_k(x))\D x\right)\stab \left(V,\int_0^\infty (1-H^{(\infty)}_k(x))\D x\right),
%&\left(V^\n,{H^\n_k}^{-1}(q)\right)\convd \left(V,{H^{(\infty)}_k}^{-1}(q)\right)
\end{align*}
%for any $q\in(0,1)$, and 
where the limit variables are finite almost surely. 
\end{lemma}
\begin{proof}
In view of~\eqref{eq:zooming_conv} we only need to establish the continuity of the respective maps.
Consider $(2k+1)$-dimensional vectors $a^\n$ and~$b^\n$ converging to some vectors $a$ and $b$, respectively, where the entries of $a^\n$ and~$a$ are non-negative and the entries of $a,b$ are finite.
Observe using~\eqref{eq:ineq} that
\begin{align*}&\sup_{x\geq 0}\left|\prod_{|i|\leq k}F\left(x+a^\n_i,b^\n_i\right)-\prod_{|i|\leq k}F\left(x+a_i,b_i\right)\right|\\
&\leq\sum_{|i|\leq k} \sup_{x\geq 0}\left|F\left(x+a^\n_i,b^\n_i\right)-F\left(x+a_i,b_i\right)\right|\to 0,\end{align*}
where convergence of $F$ is uniform in $x\geq 0$ since the limit function is continuous and non-decreasing in $x\geq 0$, and is upper bounded (Polya's theorem).
Thus, the first statement is now proven.

Concerning the second statement, we find that
\begin{align*}&\left|\int_0^\infty (1-\prod_{|i|\leq k}F(x+a^\n_i,b^\n_i))\D x-\int_0^\infty (1-\prod_{|i|\leq k}F(x+a_i,b_i))\D x\right|\\
&\leq \sum_{|i|\leq k}\int_0^\infty\left|\overline F(x+a^\n_i,b^\n_i)-\overline F(x+a_i,b_i)\right|\D x
\end{align*}
and it is left to show that each summand converges to~0, i.e.\  that the dominated convergence theorem applies.
According to Lemma~\ref{lem:bounds} both $\overline F(x+a^\n_i,b^\n_i)$ and $\overline F(x+a_i,b_i)$ are bounded by $c(1\wedge x^{-2\a-1})$, because of monotonicity of $\overline F$ in the first argument and the fact that $b_i^\n\to b_i<\infty$; the decay is even faster in the case of $\beta=\pm 1$ or when $X$ is a Brownian motion.
The proof of the second statement is now complete, since finiteness of the limit is shown in the same way.
%
%Finally, we note that the inverse of $\prod_{|i|\leq k}F(\cdot+a_i,b_i)$ is continuous on $(0,1)$, because the function itself is non-negative and strictly increasing for $x\geq 0$.
%According to~\cite[Prop.\ 0.1]{resnick} the respective inverses converge at all points in~$(0,1)$.
\end{proof}

\subsubsection{Uniform negligibility of truncation}\label{sec:uniform_neg}
Showing that truncation at a finite $k$ is uniformly negligible (in the sense of~\cite[Thm.\ 3.2]{billingsley}) is the crux of the proof.
Firstly, we will need the following representation-in-law of the sequences~$u_i^\n$, which builds on~\cite{bertoin_splitting} and self-similarity of $X$.
\begin{lemma}\label{lem:bertoin}
There exists a process $\tilde\xi$ having the law of $\xi$ and a sequence of random variables $\tau_n$ such that $(\tau_n)_{n>0}$ and $(n-\tau_n)_{n>0}$ are non-negative non-decreasing sequences and the following is true: Let \[\tilde u_i^\n:=\tilde \xi_{i+1-\{\tau_n\}}\]
for all $i\in[-\lceil\tau_n\rceil,n-\lceil\tau_n\rceil]$, and otherwise $\tilde u_i^\n:=\infty$. Then 
$(u_i^\n)_{i\in\mathbb Z}\eqd (\tilde u_i^\n)_{i\in\mathbb Z}$ for all $n\in\mathbb N_+$.
\end{lemma}
\begin{proof}
By self-similarity $(n^{1/\a}X_{t/n})_{t\in[0,n]}$ has the same law as $(X_t)_{t\in[0,n]}$.
According to~\cite{bertoin_splitting}, the law of the latter process when seen from the supremum, see~\eqref{eq:xi}, coincides with a certain process 
$\tilde\xi$ killed outside of the interval $[-\tau_n,n-\tau_n]$, where $\tau_n=\int_0^n\ind{X_t> 0}\D t$.
It is noted that $\tilde \xi$ is constructed using juxtaposition of the excursions of $X$ in half-lines according to their signs, and it does not depend on~$n$.
Clearly, $\tau_n$ and $n-\tau_n=\int_0^n\ind{X_t\leq 0}\D t$ are non-decreasing sequences going to $+\infty$, and the laws of $\tilde\xi$ and $\xi$ defined  by~\eqref{eq:xi} coincide.
It is now left to recall the definition of $u_i^\n$.
\end{proof}

We will also need asymptotic bounds on the process~$\xi$, which can be read of~\cite[Cor.\ 3.3]{fourati} or~\cite{chaumont_pardo}, see also~\cite{motoo_59} for the Brownian case.
\begin{lemma}\label{lem:LIL}
For any $p_-,p_+>0$ such that $p_-<1/\a<p_+$ it holds that 
\[\lim_{t\to\infty}\xi_t/t^{p_-}=\infty, \qquad \lim_{t\to\infty}\xi_t/t^{p_+}=0\qquad\text{ almost surely.}\]
In particular, the probability of the event 
\[E_{T,p_\pm}:=\{\forall t\geq T: \xi_t\in[t^{p_-},t^{p_+}]\}\]
tends to $1$ as $T\to\infty$.
\end{lemma}

The following result establishes convergence of certain series, which is only needed for the case of a stable process with two-sided jumps.
\begin{lemma}\label{eq:Dbound}
Assume that $\beta\in(-1,1)$ and consider 
\[D_t=\sup_{h\in[0,1]}|\xi_{t+1+h}-\xi_{t+h}|.\]
Then there exist $p_\pm>0$ such that $p_-<1/\a<p_+$ and the following series are convergent for any $T>0$:
\begin{align}
&\a\in(1,2):&\sum_{i\geq 1} i^{-2\a p_-}\e[D_i^{\a+1};E_{T,p_\pm}]<\infty,\label{eq:series1}\\
&\a\in(0,1]:&\sum_{i\geq 1} i^{-2\a p_--p_-}\e[D_i^{\a+1};E_{T,p_\pm}]<\infty. \label{eq:series2}
\end{align} 
\end{lemma}
\begin{proof}
Assume that $\a\in(1,2)$. Let us show that there exists a natural number $k$ and 
\[0=\delta_0<\delta_1<\cdots<\delta_{k-1}<1<\delta_k\] such that $\delta_{j}(\a+1)/\a-\delta_{j-1}<1$ for all $j=1,\ldots, k$. 
The $j$th inequality reads as $\delta_j<\psi(\delta_{j-1})$ with $\psi(u)=(1+u)\a/(1+\a)$ being a continuous function such that $\psi(u)>u$ iff $u<\a$.
Note that it is sufficient to pick the smallest $k$ such that
$\psi^{(k)}(0)>1$, where the latter denotes $k$th iterate. To see that such $k$ exists, simply observe that
 $\psi^{(k)}(0)$ converges to $\a>1$ as $k\to\infty$.

Choose $p_\pm$ close enough to $1/\a$ so that $\delta_{k-1}<\a p_-<1<\a p_+<\delta_k$ and 
\begin{equation}\label{eq:p_choice}\delta_{j}(\a+1)/\a-\delta_{j-1}<2\a p_--1,\qquad \text{ for all } j=1,\ldots,k.\end{equation}
According to Lemma~\ref{prop:fluct_bound} in \S\ref{sec:cond_pos}, for any $i>T$ we have
\[\p\left(\{D_i\geq i^{\delta_{j-1}/\a}\}\cap E_{T,p_\pm}\right)\leq c i^{-\delta_{j-1}},\]
because $\xi_i>i^{p_-}>i^{\delta_{j-1}/\a}$ on the respective event.
Now for any $j=1,\ldots, k$ we have
\begin{align*}&\sum_i i^{-2\a p_-}\e\left[D_i^{\a+1};\{D_i\in[i^{\delta_{j-1}/\a},i^{\delta_{j}/\a})\}\cap E_{T,p_\pm}\right]\\
&\leq\sum_i i^{-2\a p_-}i^{\delta_{j}(\a+1)/\a}\p\left(\{D_i\geq i^{\delta_{j-1}/\a}\}\cap E_{T,p_\pm}\right)\\
&\leq c\sum_i i^{-2\a p_-+\delta_{j}(\a+1)/\a-\delta_{j-1}}<\infty
\end{align*}
according to~\eqref{eq:p_choice}. Summing up over $j=1,\ldots,k$ completes the proof of~\eqref{eq:series1}, because on the event $E_{T,p_\pm}$ we have
$D_i<(i+2)^{p_+}<i^{\delta_k/\a}$ for $i>T$ large enough. Moreover, the first interval $[1,i^{\delta_1/\a})$ can be replaced by $[0,i^{\delta_1/\a})$ without any change required.
 
 Next, assume that $\a\in(0,1]$ and choose $\delta_1<\a p_-<1<\a p_+<\delta_2$. 
Similarly, to the above calculation we find that it is sufficient to  additionally guarantee that
\[-2\a p_--p_-+\delta_j(\a+1)/\a-\delta_{j-1}<-1, \qquad j=1,2.\]
This is always possible when $\delta_2<1+\delta_1\a/(\a+1)$.
\end{proof}

We are now ready to establish that truncation is indeed uniformly negligible:
\begin{lemma}\label{lem:conv}
For any $\epsilon>0$ we have 
\begin{align}
&\lim_{k\to\infty}\sup_n \p(\|H^\n-H_k^\n\|_\infty>\epsilon)= 0,\label{eq:trunc1}\\
&\lim_{k\to\infty}\sup_n \p\left(\left|\int_0^\infty (H^\n(x)-H_k^\n(x))\D x\right|>\epsilon\right)= 0,\quad\text{ for }\a\in(1,2].\label{eq:trunc2}
\end{align}
Moreover, almost surely it holds that 
\begin{align}
&\sup_{x\geq 0}\left|1-\prod_{|j|>k}F\left(x+\xi_{j+U}-V,\xi_{j+U}-\xi_{j+1+U}\right)\right|\to 0,\label{eq:Hconv_str}\\
&\int_0^\infty (1-H_k^{(\infty)}(x))\D x\to \int_0^\infty (1-H(x))\D x<\infty,\quad \text{ for }\a\in(1,2].\nonumber
\end{align}
as $k\to \infty$.
\end{lemma}
\begin{proof}
%We focus on proving~\eqref{eq:trunc_sup}, since convergence in probability of $H^{(\infty)}_k$ is simpler and can be established using some of the arguments below.
We start by showing~\eqref{eq:trunc1}.
Using~\eqref{eq:ineq} we find that 
 \[\|H^\n-H_k^\n\|_\infty\leq \sup_{x\geq 0}\sum_{|i|>k} \overline F(x+u_i^\n-V^\n,u_i^\n-u_{i+1}^\n),\]
 where the summand is 0 when either of $u_i^\n,u_{i+1}^\n$ is infinite.
% and similarly that
%  \[\left|\int_0^\infty (H^\n(x)-H_m^\n(x))\D x\right|\leq \sum_{|i|>m}\int_0^\infty|1-F(x+u_i^\n-V^\n,u_i^\n-u_{i+1}^\n)|\D x.\]
 By monotonicity of $F$ in the first argument, and the fact that $\p(V^\n>v)$ can be made arbitrarily small by choosing large enough~$v$ (recall that $V^\n\convd V$),
 it is sufficient to show that 
 \begin{equation}\label{eq:toshow1}
 \sup_n \p\left(\sum_{i>k} \overline F(\tilde u_i^\n-v,\tilde u_i^\n-\tilde u_{i+1}^\n)>\epsilon\right)\to 0,
 \end{equation}
 where we have replaced $u^\n_i$ by $\tilde u^\n_i$ having the same law as defined in Lemma~\ref{lem:bertoin}.
 Note also that the sum here runs over $i>k$ since the other part ($i<-k$) can be handled in the same way.
 
 Choose $p_\pm$ with $p_-<1/\a<p_+$ such that the conclusion of Lemma~\ref{eq:Dbound} is satisfied when $\a\in(0,2),\beta\in(0,1)$.
 Note that we may restrict to the event $\tilde E_{T,p_\pm}$ for a large enough $T>0$, see Lemma~\ref{lem:LIL}; that is, we have $t^{p_-}\leq \tilde \xi_t\leq t^{p_+}$ for all $t>T$.

First, assume that $\a\in(0,1)$ and $\beta\in(-1,1)$. According to Lemma~\ref{lem:bounds} we have the bound (this bound is 0 when $\tilde u_i^\n$ or $\tilde u_{i+1}^\n$ is infinite)
\begin{align*}
&\overline F(\tilde u_i^\n-v,\tilde u_i^\n-\tilde u_{i+1}^\n)\\
&\leq c(\tilde u_{i}^\n-v)^{-\a}(\tilde u_{i+1}^\n-v)^{-\a-1}(1\vee \tilde D_i)^{\a+1}\\
&\leq ci^{-p_-(2\a+1)}(1+\tilde D_i^{\a+1}),
\end{align*}
where $\tilde D_i\leq \sup_n|\tilde u_i^\n-\tilde u_{i+1}^\n|$, see the definition of $\tilde u_i^\n$ in Lemma~\ref{lem:bertoin}.
Now~\eqref{eq:toshow1} follows by Markov's inequality from
\[\sum_{i>k} i^{-p_-(2\a+1)}\e[\tilde D_i^{\a+1};\tilde E_{T,p_\pm}]<\infty,\]
see Lemma~\ref{eq:Dbound}.
In the case $\beta=\pm 1$ and $\a=2$ the above becomes
\[\sum_{i>k} \exp(-ci^{p_-})<\infty,\qquad \sum_{i>k} \exp(-ci^{2p_-})<\infty,\]
respectively, which is obviously true.

Next we show~\eqref{eq:trunc2}.
 With respect to the second statement we only need to show that
  \[\sup_n \p\left(\sum_{i>k}\int_0^\infty \overline F(x+\tilde u_i^\n-v,\tilde u_i^\n-\tilde u_{i+1}^\n)\D x>\epsilon\right)\to 0,\]
 In the case $\a\in(1,2),\beta\in(-1,1)$ the upper of Lemma~\ref{lem:F} reads
 \begin{equation*}
 \overline F(x+\tilde u_i^\n-v,\tilde u_i^\n-\tilde u_{i+1}^\n)<c(x+i^{p_-})^{-2\a-1}(1+D_i^{\a+1}),
 \end{equation*}
for $i>T$. Integrating over $x\geq 0$ we get the bound $ci^{-2\a p_-}(1+D_i^{\a+1})$ and the proof is again completed by the Markov's inequality and Lemma~\ref{eq:Dbound}.
In the case $\beta=\pm 1$ the bound is 
\[\sum_{i>k}\int_0^\infty\exp(-(x+i^{p_-}-v))\D x<\infty,\]
and a similar bound holds for $\a=2$.

Finally, similar (but simpler) arguments show that there is convergence in probability in~\eqref{eq:Hconv_str}. But the product is monotone for each~$x\geq 0$. Thus we have uniform convergence almost surely.
For $\a\in(1,2]$ we find using above arguments that the integral $\int_0^\infty (1-H(x))\D x$ is finite almost surely. Now the dominated convergence theorem applies.
\end{proof}

\begin{proof}[Proof of Theorem~\ref{thm:max}]
Let us show the stated properties of $H$. It is clear that $H(x),x\geq 0$ is non-decreasing and takes values in~$[0,1]$. Moreover, $H(0)=0$ since one of the terms in the product is~0.
Observe that~\eqref{eq:Hconv_str} implies convergence of $H_{k}^{(\infty)}$ to $H$ uniformly in $x\geq 0$ on the set of probability~1.
Thus $H$ is continuous and $H(\infty)=1$, because the same is true about $H_{k}^{(\infty)}$. 
Finally, $H$ is strictly monotone, since $H(x)>0$ for every $x>0$ which follows from positivity of $H_k^{(\infty)}$ and~\eqref{eq:Hconv_str}.

Stable convergence statements in~\eqref{eq:Hconv} and~\eqref{eq:T1} follow from Lemma~\ref{lem:conv_trunc} and Lemma~\ref{lem:conv} by means of~\cite[Thm.\ 3.2]{billingsley} extended to the setting of stable convergence.
Concerning~\eqref{eq:T2} we apply Skorokhod's representation theorem to the sequence $H^\n$ (the underlying space of continuous functions with a limit at $\infty$ is indeed separable, as it can be time-changed into the space of continuous functions on~$[0,1]$).
The inverse $H^{-1}$ is continuous and finite on $(0,1)$ and hence we have convergence of respective inverses~\cite[Prop.\ 0.1]{resnick}.
 \end{proof}
 
% TODO: Is the space of continuous cdf on $[0,\infty)$ separable?
\subsection{Related results}
Here we provide the proofs (or just the main ingredients) of the results related to Theorem~\ref{thm:max}.
\subsubsection{Linear Brownian motion} \label{secA.4.1}
\begin{proof}[Proof of Corollary~\ref{cor:BM}]
%Here we prove the statement of~\S\ref{sec:BM} concerning linear Brownian motion.
Firstly, note that the scaling $\sigma$ can be indeed taken out as in~\eqref{eq:BM1} and~\eqref{eq:BM2} . 
This is true in general, because we may always rescale the process and the corresponding observations before the analysis.
Thus we may assume that $\sigma=1$ in the following.

Now suppose that $\mu\neq 0$ and so $X$ is not self-similar. Recall that the estimators are the same as in the case~$\mu=0$.
Furthermore, according to~\cite{asmussen_glynn_pitman1995} the convergence in~\eqref{eq:zooming_conv} is still true, where the limit variables are defined in terms of the same 3-dimensional Bessel process.
The main difficulty is that Lemma~\ref{lem:bertoin} is no longer true and the proof of uniform negligibility of truncation fails.

  By Girsanov's theorem, we may introduce arbitrary drift using exponential change of measure $\D\p'/\D\p=\exp(a X_1+b)$ with appropriately chosen constants $a,b\in\R$. But then 
 \begin{align*}&\p'(\|H^\n-H_k^\n\|_\infty>\epsilon)=\e[\exp(a X_1+b);\|H^\n-H_k^\n\|_\infty>\epsilon]\\&
 \leq \exp(|a|c+b)\p(\|H^\n-H_k^\n\|_\infty>\epsilon)+\p'(|X_1|>c),\end{align*}
 where $c>0$ is arbitrary. But as $k\to \infty$ the $\limsup_n$ of this expression converges to $\p'(|X_1|>c)$, which can be made arbitrarily small.
 Thus~\eqref{eq:trunc1} holds for an arbitrary linear Brownian motion, and the same argument works for~\eqref{eq:trunc2}.
 \end{proof}

\begin{proof}[Proof of Lemma~\ref{lem:UI}]
	It is only required to show that 
	$\e[\left(\sqrt n|\overline X_1-\overline T_n^\text{\rm mean}|\right)^p]$
	is bounded for an arbitrarily large $p$ and all~$n$. 
	Furthermore, we may again restrict our attention to a driftless Brownian motion by change of measure and Cauchy-Schwarz inequality.
	The fact that $\e [\exp(\theta V^\n)]$ for any $\theta$ is bounded was established in~\cite{asmussen_glynn_pitman1995}, and so it is sufficient to show that
	\begin{align*}&\e\left[\left(\int_0^\infty(1-H_n(xn^{-1/2}))\D x\right)^p\right]\\
	&\leq \e\left[\left( \sum_i\int_0^\infty\overline F(x+u_i^\n-V^\n,u_i^\n-u_{i+1}^\n)\D x\right)^p\right]\end{align*}
	is bounded. The right-hand side is increased by pulling the sum out. Using the explicit expression for $\overline F$ we see that it is left to consider
	\begin{align*}&\sum_{i\geq 1} \e \left[\left(\int_0^\infty \exp\left(-2(x+u_i^\n\wedge u_{i+1}^\n-V^\n)^2\right)\D x\right)^p\right]\\
	&\leq c\sum_{i\geq 1}\e \left[\exp\left(-p(u_i^\n\wedge u_{i+1}^\n-V^\n)\right)\right],\end{align*}
	where we used that $\overline \Phi(4x)<c\exp(-x)$. Moreover, $V^\n$ can be dropped out, because of Cauchy-Schwarz inequality and boundedness of $\e [\exp(p V^\n)]$. Finally, use Lemma~\ref{lem:bertoin} to get the bound:
	\begin{align*}&\sum_{i\geq 1} \e[\exp(-p\min_{t\in[i,i+2]}\xi_t)].
	\end{align*}
	The above is bounded by
	\[\sum_{i\geq 1} \e[\exp(-p\xi_i/2)]+\sum_{i\geq 1} \p(\min_{t\in[i,i+2]}\xi_t<\xi_i/2).\]
	The first sum is finite, because the inequality between arithmetic and quadratic means, $\sqrt{a^2+b^2+c^2}\geq (|a|+|b|+|c|)/\sqrt 3$, and the definition of Bessel-3 process imply that
		the respective terms are bounded by the quantity $\e [\exp(-p\sqrt i|Z|/(2\sqrt 3))^3]$
	where $Z$ is standard normal. By Tauberian theorem this quantity behaves as $\ell i^{-3/2}$ for large $i$ with $\ell$ being a positive constant, and the first sum is indeed finite.
	The second sum can be treated using the arguments from Appendix~\ref{sec:cond_pos}. In particular, we can show that $\p_x^\uparrow(\underline X_2<x/2)<c\exp(-x)$, and hence we are left to consider $\sum_i\exp(-\xi_i/2)$ again. The proof is now complete.
\end{proof}

 \subsubsection{Joint estimation: proof of Corollary~\ref{cor:joint}}
 %The analogue of $\xi$ process for $-X$ is the reversed process $(\xi_{(-t)-})_{t\in\R}$ and thus the limit result for the estimator $\underline T^\text{mean}_n$ must be the same.
 The only new ingredient needed is the joint convergence of sequences in~\eqref{eq:zooming_conv}
 corresponding to the processes $X$ and $-X$ to their respective limits which are independent. Similar result appears in~\cite[Lem.~1]{asmussen_ivanovs_twosided} and only a minor adaptation is needed.
\subsubsection{On simplified estimators: proof of Corollary~\ref{cor:truncation}}
We only need to show that the analogue of~\eqref{eq:zooming_conv} is true, where we take the respective $2k+1$ elements in the vectors on the left.
One can not apply the continuous mapping theorem for the infinite sequences though. 
We consider truncated sequences, apply the continuous mapping theorem, and then show uniform negligibility of truncation.
The latter follows from the fact that 
\[\lim_{T\to\infty}\sup_n\p\left(\sup_{|t|>T} \xi^\n_t<a\right)=0\]
for any $a>0$, which readily follows from the representation of~$\xi^\n$ as in Lemma~\ref{lem:bertoin} in the self-similar case.

\subsubsection{Unknown parameters: proof of Proposition~\ref{prop:cip0_max}}
We will show that $n^{1/\a}(\widetilde T_n^\mean-\overline T_n^\mean)\cip 0$ when $\a\in(1,2]$, and the same is true for the conditional median estimator for all $\a\in(0,2]$. The proof of continuity of the limit disributions follows similar steps, see also~\cite{ivanovs_zooming} for the convergence of the respective processes $\xi$.
The above readily translates into
\[\int_0^\infty (H^{\theta_n}_n(xn^{-1/\a})-H_n(xn^{-1/\a}))\D x\cip 0,\qquad \sup_{x\geq 0}|H^{\theta_n}_n(xn^{-1/\a})-H(x)|\cip 0,\]
respectively.
We focus on the class of strictly stable L\'evy processes (the proof for the class (i) is similar but easier) and let $X^n$ be the process with parameters~$\theta_n$. Furthermore we write $F^n$ and $f^n$ for the analogues of conditional distribution $F$ and density $f$.

We claim that it is sufficient to establish that $F^n$ converges to $F$ continuously, i.e.\
\begin{equation}\label{eq:contConv}
F^n(x_n,y_n)\to F(x,y)\qquad\text{for any }(x_n,y_n)\to (x,y), \text{ s.t. }x>y_+. 
\end{equation}
For this note that $\a_n$ is arbitrarily close to $\a$ with high probability, and thus the arguments from the proof of Theorem~\ref{thm:max} apply 
essentially without a change. 

Thus we are left to prove~\eqref{eq:contConv} by reexamining the proof of Lemma~\ref{lem:F}.
Firstly, we observe that $(X_{\tau^+_{x_n}},\tau_{x_n}^+)\ind{\tau^+_{x_n}<\infty}$ under $\p_{\theta_n}$ weakly converges to the respective quantity under~$\p$,
which follows by the (generalized) continuous mapping theorem and weak convergence of the L\'evy processes.
Secondly, the function 
\[g_n(t,x,y):=f^n_{1-t}(y-x)=(1-t)^{-1/\a_n}f^n((1-t)^{-1/\a_n}(y-x))\]
converges to the obviously defined $g(t,x,y)$ continuously on the domain $t\in(0,1),x\geq 0,y\in\R$, which follows from continuous convergence of the density $f^n$ of $X_1^n$, see Lemma~\ref{lem:density} below. 
Hence we have weak convergence of the quantity under the expectation in~\eqref{eq:F1}, and so it is left to show that the respective quantities are bounded. Lemma~\ref{lem:density} completes the proof.

\begin{lemma}\label{lem:density}
There is the uniform convergence: $\sup_{x\in\R}|f^n(x)-f(x)|\to 0$ as $n\to\infty$.
Moreover, for any $\epsilon>0$ it holds that 
\[\sup_n\sup_{t\in(0,1),x\geq \epsilon}f^n_t(x)<\infty.\]
\end{lemma}
\begin{proof}
The characteristic function of $X^n_t$ is given by $\exp(-c^\pm_n|z|^{\a_n} t)$ according to $\pm z>0$ with $c^\pm_n$ being a complex constant with positive real part (converging to $c^\pm$), see~\cite[Thm.\ C.4]{zolotarev_book}.
Thus by inversion formula we have
\[\sup_{x\in\R}|f^n(x)-f(x)|\leq \frac{1}{2\pi}\int |\exp(-c^\pm_n|z|^{\a_n} t)-\exp(-c^\pm|z|^{\a} t)|\D z,\]
but this converges to 0 by the dominated convergence theorem, since the real parts of $c_n^\pm$ are positive and bounded away from~0.

With respect to the second statement we need to show that
\[\int_1^\infty \exp(-iz x-c_nz^{\a_n} t) \D z\]
is bounded for all $t\in(0,1),x\geq \epsilon$ and all~$n$, where $c_n=c_n^+$; the integral over $(-\infty,-1]$ is hadled in the same way, whereas the rest is clearly bounded by~2.
Using integration by parts we find that it is sufficient to show that
\[\int_1^\infty \frac{c_n t\a_nz^{\a_n-1}}{i x}\exp(-iz x-c_nz^{\a_n} t) \D z\]
is bounded, or equivalently the boundedness of
\[\int_1^\infty \a_nz^{\a_n-1}t\exp(-r_nz^{\a_n} t) \D z=\int_{t}^\infty \exp(-r_nz) \D z\leq \frac{1}{r_n},\]
where $r_n=\Re(c_n)$. But $r_n\to r>0$ and we are done.
\end{proof}

%===========================================

\section{Proofs for local and occupation times}\label{sec:proofs_loc}
Here $X$ denotes a linear Brownian motion with drift parameter $\mu\in\R$ and scale $\sigma>0$. 
\begin{proof}[Proof of Lemma~\ref{lem:formulae}]
	The fact that $\e\left[L_t(x)| X_{t}=z\right] $ does not depend on~$\mu$ follows readily by applying exponential change of measure, for example. Thus we may assume that $\mu=0$ and consider the process $\sigma X$ with $X$ being the standard Brownian motion.
	Using self-similarity of $X$ we find
	\begin{align*}
	&\left(\frac{1}{2\epsilon}\int_0^t 1_{(x-\epsilon,x+\epsilon)}(\sigma X_s)\D s,\sigma X_t\right)=\left(\frac{t}{2\epsilon}\int_0^1 1_{(x-\epsilon,x+\epsilon)}(\sigma X_{ts})\D s,\sigma X_t\right)\\
	&\stackrel{d}{=}\left(\frac{t}{2\epsilon}\int_0^1 1_{(x-\epsilon,x+\epsilon)}(\sigma\sqrt t X_{s})\D s,\sigma\sqrt t X_1\right)	
	\end{align*}
	and we readily find the stated expression for $\e\left[L_t(x)| X_{t}=z\right]$ from the definition of~$L$. For further reference let us also note that
	\begin{equation}\label{eq:eq:loc_scaling}	(L_t(x),X_t)	\stackrel{d}{=}(\sqrt tL_1(x/\sqrt t),\sqrt t X_1)\qquad \text{under }\p^0.\end{equation}
	The formula for $\e\left[O_t(x)| X_{t}=z\right]$ is obtained similarly, or directly from~\eqref{eq:occ_density}.
	
	Next, we note that $g(x,z)=g(-x,-z), G(x,z)=1-G(-x,-z)$ follow easily from symmetry, and so we assume in the following that $x\geq 0$.
	From~\cite[1.3.8]{borodin_salminen} we find 
	\[g(x,z)=\exp(z^2/2)\int_0^\infty y(|z-x|+|x|+y)\exp(-(|z-x|+|x|+y)^2/2)\D y\]
	which indeed evaluates to the given expression. Next, 
	we recall the Mill's ratio: $\overline\Phi(z)/\varphi(z)\sim 1/z$ as $z\to\infty$. Hence 
	\begin{equation}\label{eq:f1_asy}g(x,z)\sim \frac{1}{|z-x|+|x|}\exp(-(|z-x|+|x|)^2/2+z^2/2)\qquad\text{ as }|x|\vee|z|\to\infty,\end{equation}
	showing that $g(x,z)$ is bounded since $|z-x|+|x|\geq |z|$.
	
	Finally, $G(x,z)$ is clearly bounded by~1 and the given formulae are found from the occupation density formula $G(x,z)=\int_x^\infty g(y,z)\D y$, see~\eqref{eq:occ_density}.
	%Alternatively, we may note that
	%\[F(x,z)=\int_0^1\e(X_t>x|X_1=z)\D t=\int_0^1\overline \Phi\left(\frac{x-tz}{\sqrt{t(1-t)}}\right)\D t,\]
	%since $X_t|X_1=z$ is normal with mean $tz$ and variance $t(1-t)$.
\end{proof}
\subsection{Local time}
\begin{proof}[Proof of \eqref{eq:loc_thm}]
Firstly, we may replace $t$ by $\lfloor tn\rfloor/n$ on the left hand side of~\eqref{eq:loc_thm}, see~\cite[Rem.\ 2]{Jacod98}.
The result would follow from~\cite[Thm.\ 2.1]{Jacod98} if we show that $g$ satisfies condition~\cite[(B-$r$)]{Jacod98} for some $r>3$.
But this follows from the bound $g(x,z)< c\exp(-2|x|+2|z|)$ for all $x,z\in\R$, see the proof of Lemma~\ref{lem:jacod_cond}.
%
%Note that $f(x,z)=f(-x,-z)$ and so we may assume that $x\geq 0$. If $z\geq x$ then $|z-x|+|x|=z$ and otherwise $|z-x|+|x|=2x-z>x$.
%From~\eqref{eq:f1_asy} we have an upper bound:
%\[f(x,z)<C(\ind{z\geq x}+\ind{z<x}e^{-2(x-z)x})\]
%for some $C>1$. 
%For $z\geq x$ we indeed have $1\leq e^{-2|x|}e^{2|z|}$. For $x\geq z+1$ we find that $e^{-2(x-z)x}\leq e^{-2x}\leq e^{-2|x|}e^{2|z|}$.
%Finally, for $z\in(x-1,x)$ we have $e^{-2(x-z)x}\leq 1$, whereas $e^{-2|x|}e^{2|z|}\geq e^2$. Hence we may simply increase $C$ in the above stated bound accordingly.

Now we have the stated convergence, but the constant in front of the limit needs to be identified. The expressions in~\cite{Jacod98} are lengthy and non-trivial to evaluate, because of the generality assumed therein. In our case, $g(x, X_1) =\e[L_1(x)|X_1]$ is the conditional expectation and, in fact, a rather short direct proof can be given yielding the constant.

{\bf Direct Proof:} 
  As in~\cite{Jacod98}  we observe that it is sufficient to consider the case $\mu=0$, which can be extended to an arbitrary $\mu$ using change of measure argument.
  Importantly, $\widehat L_t(x)$ is a functional of $X$ and this functional does not depend on~$\mu$.
  Next, consider a standard Brownian motion $X^0_t=X_t/\sigma$ and assume that our result is proven for~$X^0$. Noting that $L_t(x)=\frac{1}{\sigma}L^0_t(x/\sigma)$ as well as $\widehat L_t(x)=\frac{1}{\sigma}\widehat L^0_t(x/\sigma)$ we find that
  \begin{align*}n^{1/4} \left(\widehat{L}_t (x)- L_t(x) \right)&=\frac{1}{\sigma}n^{1/4} \left(\widehat{L}^0_t (x/\sigma)- L^0_t(x/\sigma) \right)\\
   &\stab \frac{v_l}{\sigma}W_{L^0_t(x/\sigma)}= \frac{v_l}{\sigma}W_{\sigma L_t(x)}.\end{align*} 
   It is left to replace the process $W_{\sigma t}$ by $\sqrt\sigma W_t$ having the same law. Thus we may assume in the following that $X$ is a standard Brownian motion.

	Let $S_t^n=  \sum_{i=1}^{\lfloor tn\rfloor}  \xi_{in}$ be the pre-limiting object, where 
	\[
	\xi_{in}= n^{-1/4} \left( g(\sqrt n(x-X_{\frac{i-1}{n}}), \sqrt{n} \Delta_i^n X) - \sqrt n L_{[\frac{i-1}{n},\frac{i}{n}]}(x)  \right),
	\]
	$\Delta_i^n X= X_{i/n} - X_{(i-1)/n}$ and $L_{[a,b]}(x)$ denotes the local time at $x$ in the interval $[a,b]$.
	Firstly, observe using the scaling property~\eqref{eq:eq:loc_scaling} that 
\[h_1(x):=\e\left( g(x, \sqrt{n} X_{1/n}) - \sqrt n L_{1/n}(x/\sqrt n)  \right)=\e\left( g(x, X_1) - L_1(x)  \right)=0.\]
Thus we have 	
	\[\e[\xi_{in}|\F_{\frac{i-1}{n}}]=n^{-1/4}h_1(\sqrt n(x-X_{\frac{i-1}{n}}))=0,\] 
	and similarly we find that
	\begin{align*}\label{eq:h}
	&\e[ \xi_{in}^2|  \F_{\frac{i-1}{n}}]  = n^{-1/2} h_2(\sqrt n(x-X_{\frac{i-1}{n}})),\\% &h_1(y) = \e(g(y, X_1) - L_{[0,1]}(y))^2,\\
	&\e[ \xi_{in} \Delta_i^n X| \F_{\frac{i-1}{n}}] = n^{-3/4} h_3(\sqrt n(x-X_{\frac{i-1}{n}}))=0,\\%, &h_2(y) = \e[(g(y, X_1) - L_{[0,1]}(y))X_1],\nonumber\\
	&\e[ \xi_{in}^4|  \F_{\frac{i-1}{n}}]  = n^{-1} h_4(\sqrt n(x-X_{\frac{i-1}{n}})),% &h_3(y) = \e(g(y, X_1) - L_{[0,1]}(y))^4.\nonumber
	\end{align*}
	where $h_i(y)=\e(g(y, X_1) - L_1(y))^i$ for $i=2,4$, and $h_3(y)=\e[(g(y, X_1) - L_1(y))X_1]=0$.
	
	Let us show that $h_i$ for $i=2,4$ are bounded and in $L^1(\R)$. By Minkowski's and Jensen's inequality we have the bound
	$h_i(y)\leq 2^i \e[ L_1(y)^i]$.	Using additivity of $L$ we deduce that 
	\[
	\e[ L_1(y)^i] \leq \p(\tau_y<1)\e[L_1(0)^i],
	\]
	where the latter moment is finite and $\tau_y$ is the first passage time of $X$ into the level~$y$. Finally, note that 
	\[\int_0^\infty \p(\tau_y<1)\D y=\int_0^\infty \p(\overline X_1>y)\D y=\e \overline X_1<\infty\]
	and hence by symmetry $h_i(y)$ are integrable.
	Thus according to~\cite[Thm.\ 1.1]{Jacod98} we have
	\[n^{-1/2}\sum_{i=1}^{\lfloor nt\rfloor} h_i(\sqrt n(x-X_{\frac{i-1}{n}}))\cip L_t(x)\int h_i(x)\D x, \qquad i=2,4,\]
	where the convergence is uniform on compact intervals of time.
This immediately yields that 
	\begin{align} 
&\sum_{i=1}^{\lfloor nt\rfloor}   \e[ \xi_{in}^2|  \F_{\frac{i-1}{n}}] \cip v_l^2 L_t(x),\qquad \sum_{i=1}^{\lfloor nt\rfloor}   \e[ \xi_{in} \Delta_i^n X| \F_{\frac{i-1}{n}}] = 0,\\
	&\sum_{i=1}^{\lfloor nt\rfloor}   \e[ \xi_{in}^2 1_{\{|\xi_{in}|>\epsilon\}}| \F_{\frac{i-1}{n}}]\leq \epsilon^{-2} \sum_{i=1}^{\lfloor nt\rfloor} \e[ \xi_{in}^4 | \F_{\frac{i-1}{n}}] \cip 0 \qquad \text{ for any }\epsilon>0. \nonumber
	\end{align}
Finally, let $N$  be a continuous bounded martingale orthogonal to $X$, i.e. $[X,N]=0$. For $t\geq (i-1)/n$ define the process 
$M_t=\e[\xi_{in}| \F_t]$. Then the martingale representation theorem implies the existence of a progressively measurable process $\eta^n$ 
such that
\[ M_t = \int_{\frac{i-1}{n}}^t \eta_s^n dX_s.\]
Since $[X,N]=0$ we conclude that 
\begin{align} 
\e[\Delta_i^n N \xi_{in}| \F_{\frac{i-1}{n}}] = \e[\Delta_i^n N \Delta_i^n M| \F_{\frac{i-1}{n}}]=0.
\end{align}	
The result now follows from~\cite[Thm.\ 7.28]{jacod_shiryaev}. Moreover, we have a simple expression for $v_l^2=\int h_2(y)\D y$ which is evaluated in Lemma~\ref{lem:v} below.
\end{proof}

%\begin{lemma}\label{lem:h} In the case of standard Brownian motion the functions $h_i(y)=\e(g(y, X_1) - L_{[0,1]}(y))^i$ for $i=2,4$ are bounded and in $L^1(\R)$.
%\end{lemma}
%\begin{proof}
%	Recalling $g(x, X_1) =\e[L_{[0,1]}(x)|X_1]$ we find that
%	\[h_2(y) \leq \e L_{[0,1]}(y)^2, \qquad h_4(y)\leq 2^4 \e L_{[0,1]}(y)^4,\]
%	where we used Minkowski's and Jensen's inequality for the second bound.
%	Using additivity of the local time we deduce that 
%	\[
%	\e[ L_{[0,1]}(y)^i] \leq \p(\tau_y<1)\e[L_{[0,1]}(0)^i],
%	\]
%	where the latter moment is finite and $\tau_y$ is the first passage time of $X$ into the level~$y$. Finally, note that 
%	\[\int_0^\infty \p(\tau_y<1)\D y=\int_0^\infty \p(\overline X_1>y)\D y=\e \overline X_1<\infty\]
%	and the proof is complete by symmetry.
%\end{proof}

It is left to calculate~$v_l^2$, which is the integrated reduction in variance when $L_1(y)$ is replaced by its conditional mean $\e[L_1(y)|X_1]$:
\begin{lemma}\label{lem:v}
	For a standard Brownian motion we have
		\[\int_{\R} \e[\left(g(y, X_1) - L_1(y)\right)^2] \D y=2\frac{3\log(1+\sqrt 2)-\sqrt 2}{3\sqrt\pi}.\] 
\end{lemma}
\begin{proof}
	Recalling that $g(y,X_1)=\e[L_1(y)|X_1]$ we find
\begin{align*}
\int_\R \left(\e [L^2_1(y)]-\e [f^2(y,X_1)]\right)\D y=2\int_0^\infty\left(\e [L^2_1(y)]-\e [f^2(y,X_1)]\right)\D y.
\end{align*}
According to~\cite[1.3.4]{borodin_salminen} we calculate
\[\int_0^\infty\e [L^2_1(y)]\D y=\int_0^\infty\int_0^\infty x^2\sqrt\frac{2}{\pi}\exp(-(x+y)^2/2)\D x\D y=\frac{2}{3}\sqrt\frac{2}{\pi},\]
and
\[\int_0^\infty\e [f^2(y,X_1)]\D y=\int_0^\infty\int_\R\overline \Phi^2(|z-y|+y)/\varphi(z)\D z\D y=\frac{\sqrt 2-\log(1+\sqrt 2)}{\sqrt \pi},\]
where in both cases we first integrate in $y>0$. %The latter integral is evaluated by splitting it into two parts: $z>x$ and $z<x$, and first integrating over $x$ using integration by parts.
Combine these formulae to get the result.%~$v^2=2\frac{3\log(1+\sqrt 2)-\sqrt 2}{3\sqrt\pi}$.
\end{proof}

\subsection{Occupation time}
\begin{proof}[Proof of~\eqref{clton}]
We may assume that $\mu=0$ and let $X^0_t=X_t/\sigma$.
Supposing that the result is true for $X^0$ we get
\[n^{\frac 34} \left(\widehat O_t(x) -  O_t(x) \right)=n^{\frac 34} \left(\widehat O^0_t(x/\sigma) -  O^0_t(x/\sigma) \right)\to v_oW_{L^0_t(x/\sigma)}=v_oW_{\sigma L_t(x)}\]
and so we assume that $X$ is a standard Brownian motion in the following.
 
Letting
\[
\xi_{in}= n^{-\frac 14} \left(
G \left(\sqrt{n} (x-X_{\frac{i-1}{n}}), \sqrt{n}  \Delta_i^n X \right)   - n \int_{\frac{i-1}{n}}^{\frac{i}{n}} 1_{(x,\infty)} (X_s) ds \right)
\]
and using
\[(nO_{1/n}(x/\sqrt n),\sqrt n X_{1/n})\stackrel{d}{=}(O_1(x),X_1)\]
we find that
\begin{align*}
	&\e[ \xi_{in}^2|  \F_{\frac{i-1}{n}}]  = n^{-1/2} h_2(\sqrt n(x-X_{\frac{i-1}{n}})),\\
	&\e[ \xi_{in} \Delta_i^n X| \F_{\frac{i-1}{n}}] = 0,\\	
	&\e[ \xi_{in}^4|  \F_{\frac{i-1}{n}}]  = n^{-1} h_4(\sqrt n(x-X_{\frac{i-1}{n}})),
	\end{align*} 
where $h_j(y) = \e[G(y,X_1) - O_1(y)]^j$	for $j=2,4$. 

It is left to prove that $h_j$ are bounded and in $L^1(\mathbb R)$ for $j=2,4$. The result then follows from~\cite[Thm.\ 1.1]{Jacod98} and~\cite[Thm.\ 7.28]{jacod_shiryaev} as for the local time.
It would be sufficient to show the same property for $\e [(O_1(y)-c_y)^j]$ where $c_y$ is arbitrary, because $G(y,X_1)-c_y$ is the conditional expectation of $O_1(y)-c_y$ given $X_1$.
When $y\geq 0$ we take $c_y=0$ and observe that $\e [O_1(y)^j]\leq \p(\tau_y<1)$ which is bounded and integrable over $[0,\infty)$, see the local time case.
When $y<0$ we take $c_y=1$ and observe that $\e [(1-O_1(y))^j]\leq \p(\tau_y<1)$ and the same conclusion is true.
The proof is complete upon calculation of $v_o^2$ which is given in Lemma~\ref{lem:v_occ} below.
\end{proof}

\begin{lemma}\label{lem:v_occ}
	For a standard Brownian motion we have
		\[\int_{\R} \e\left[G(y, X_1) - O_1(y)\right]^2 \D y=\frac{13\sqrt 2-15\log(1+\sqrt 2)}{45\sqrt \pi}.\] 
\end{lemma}
\begin{proof}
Note that
\[\int_{\R} \e\left[G(y, X_1) - O_1(y)\right]^2 \D y=2\int_0^\infty(\e O_1(y)^2-\e G(y,X_1)^2)\D y,\]
because for $y<0$ the integrand can be rewritten as $\e [(1-O_1(y))^2]-\e [(1-G(y,X_1))^2]$ corresponding to the occupation time in $(-\infty,y)$ and its conditional expectation, and it is left to apply symmetry.
 
 The density of the occupation time $O_1(y)$ is given in~\cite[1.4.4]{borodin_salminen} and reads as
 \[\frac{1}{\pi\sqrt{x(1-x)}}\exp\left(-\frac{y^2}{2(1-x)}\right), \qquad x\in(0,1).\]
 Thus we find $\int_0^\infty\e [O_1(y)^2]\D y=\frac{\sqrt 2}{5\sqrt \pi}$ by integrating in~$y$ first.
 
 Similar trick works in the calculation of
 \[\int_0^\infty \e [F^2(y,X_1)]\D y=\frac{\sqrt 2+3\log(1+\sqrt 2)}{18\sqrt \pi}.\]
 Combination of these expressions yields the result.
\end{proof}

\subsection{Unknown parameters} \label{unknownpar}
Let us define $g_\sigma(x,z)=\frac{1}{\sigma}g(x/\sigma,z/\sigma)$ together with $G_\sigma(x,z)=G(x/\sigma,z/\sigma)$. 
%The following lemma verifies the condition of~\cite{Jacod98} for the derivatives of $f_\sigma$ and $G_\sigma$.
\begin{lemma}\label{lem:jacod_cond}
For any $\sigma_0>0$ there exist constants $\epsilon\in(0,\sigma_0)$ and $c,a>0$ such that
\[\sup_{\sigma\in[\sigma_0-\epsilon,\sigma_0+\epsilon]}\left|\frac{\partial g_\sigma(x,z)}{\partial\sigma}\right|\vee\left|\frac{\partial G_\sigma(x,z)}{\partial\sigma}\right|\leq c \exp(a (|z|-|x|))\]
for all $x,z\in\R$.
\end{lemma}
\begin{proof}
Recall that $g(x,z)=g(-x,-z),G(x,z)=1-G(-x,-z)$ and so we may assume that $x\geq 0$. 
Furthermore, it is sufficient to establish the stated property for $\partial g_\sigma/\partial\sigma$. This is so, because 
$G_\sigma(x,z)=\int_x^\infty g_\sigma(y,z)\D y$, the derivative $\partial g_\sigma/\partial\sigma$ is continuous in $\sigma$ away from~$0$ and integrable in $y\geq 0$.
Hence 
\[\frac{\partial  G_\sigma(x,z)}{\partial \sigma}=\int_x^\infty \frac{\partial  g_\sigma(y,z)}{\partial \sigma}\D y\leq c\int_x^\infty \exp(-a y) \D y \exp(a|z|)\]
and the bound follows.

It is sufficient to establish the bound for $x\geq 0$:
\[\left|\frac{\partial  g(x/\sigma,z/\sigma)}{\partial \sigma}\right|\leq c(\exp(a(|z|-x))\wedge 1)\]
locally uniformly in $\sigma>0$. This is so, because $g(x/\sigma,z/\sigma)/\sigma^2$ satisfies the analogous bound, see~\eqref{eq:f1_asy}.

Writing $x',z'$ for $x/\sigma,z/\sigma$, respectively, we find from Lemma~\ref{lem:formulae} for $z\geq x$ that
\[\partial g(x/\sigma,z/\sigma)/\partial \sigma=z'\frac{\varphi(z')-z'\overline\Phi(z')}{2\sigma\varphi(z')}=:h(z').\]
By L'H\^{o}pitale and Mill's ratio this quantity tends to~0 as $z'\to \infty$, and thus this quantity is bounded for all $z\geq x\geq 0$ locally uniformly in $\sigma>0$.

Next, we consider $z<x$ where
\begin{align*}
&\partial g(x/\sigma,z/\sigma)/\partial \sigma=\frac{(2x'-z')\varphi(2x'-z')-{z'}^2\overline\Phi(2x'-z')}{2\sigma\varphi(z')}\\
&=\frac{2x'(x'-z')}{\sigma(2x'-z')}\exp(-2x'(x'-z'))+\frac{{z'}^2}{(2x'-z')^2}h(2x'-z')\exp(-2x'(x'-z')).
\end{align*}
Note that $2x'-z'>(x'-z')\vee |z'|$ and so the above terms stay bounded when $2x'-z'\to 0$ implying that $x',z'\to 0$.
Moreover, ${z'}^2/(2x'-z')^2$ is bounded and so it is left to consider $(1+2x'(x'-z'))\exp(-2x'(x'-z'))$ as $x'\to \infty$.
For $x'>z'+1$ this is bounded by $c\exp(-x')$ and otherwise by $c$, which is sufficient.
\end{proof}

\begin{proof}[Proof of Proposition~\ref{prop:cip0}]
Observe that
\begin{align*}
&n^{1/4}\sup_{t\leq T}\left|\widehat L_t(x)-\widetilde L_t(x)\right|\\ &\leq n^{-1/4}\sum_{i=1}^{\lfloor nT\rfloor}\left|g_\sigma(\sqrt n(x-X_{\frac{i-1}{n}}),\sqrt n\Delta_i^n X)-g_{\sigma_n}(\sqrt n(x-X_{\frac{i-1}{n}}),\sqrt n\Delta_i^n X)\right|.
\end{align*}
According to~\eqref{eq:rate} we may assume that $n^{1/4}|\sigma_n-\sigma|<h$ for an arbitrary $h>0$ and all large $n$.
By mean value theorem and Lemma~\ref{lem:jacod_cond} we have an upper bound
\begin{align*}&n^{-1/4}\sum_{i=1}^{\lfloor nT\rfloor}|\sigma_n-\sigma|\tilde g(\sqrt n(x-X_{\frac{i-1}{n}}),\sqrt n\Delta_i^n X)\\
&\leq hn^{-1/2}\sum_{i=1}^{\lfloor nT\rfloor}\tilde g(\sqrt n(x-X_{\frac{i-1}{n}}),\sqrt n\Delta_i^n X),
\end{align*}
where $\tilde g(x,z)=c \exp(-a |x|+a|z|)$. 
But $\tilde g$ verifies condition (B-0) in~\cite{Jacod98} and thus our upper bound converges to $hL$ in probability,
where $L$ is a certain finite random variable, see~\cite[Thm.\ 1.1]{Jacod98}.
The proof is complete since $h>0$ can be arbitrarily small. The corresponding proof for the occupation time measure follows exactly the same arguments. 
\end{proof}

\section{On $X$ conditioned to stay positive}\label{sec:cond_pos}
Throughout this section we assume that $\a\in(0,2)$ and $\beta\neq \pm 1$.
Let us recall that $(\xi_{(-t)-})_{t\geq 0}$ is a Feller process and, as usual, we denote its law when started from $x>0$ by $\p^\uparrow_x((X_t)_{t\geq 0}\in\cdot)$. Such a process can be seen as $X$ conditioned to stay positive in a certain limiting sense, see~\cite{chaumont,caballero2006conditioned} for the basic properties of this process. The law of $(\xi_t)_{t\geq 0}$ is then $(-X)$ conditioned to stay positive, and the following bound holds without a change.
\begin{proposition}\label{prop:fluct_bound}
There exists a constant $c>0$ such that for all $x,v>0$ with $x>v$ we have
\begin{align*}
\p^\uparrow_x(\sup_{h\in[0,1]}|X_{1+h}-X_h|>v)<cv^{-\a}.
\end{align*}
\end{proposition}
The proof will be at the end of this section.
Let us note that the restriction $x>v$ can not be removed in the above bound. We start with a simpler result where $h=0$:
\begin{lemma}\label{lem:cond_bound}
There exists $c>0$ such that for all $x>v>0$ we have
\[\p^\uparrow_x(|X_1-x|>v)<cv^{-\a}.\]
\end{lemma}
\begin{proof}
Let $\rho=\p(X_1<0)$ be the negativity parameter.
Recall the semigroup of the conditioned process~\cite{caballero2006conditioned}:
\[\p^\uparrow_x(X_1\in\D y)=\frac{y^{\a\rho}}{x^{\a\rho}}\p_x(X_1\in\D y,\underline X_1>0).\]
Hence
\begin{align*}
\p^\uparrow_x(|X_1-x|>v)&=\frac{1}{x^{\a\rho}}\e_x[X_1^{\a\rho};|X_1-x|>v,\underline X_1>0]\\
&\leq \frac{1}{x^{\a\rho}}\e[(X_1+x)^{\a\rho};|X_1|>v,X_1>-x]\\
&=\frac{1}{x^{\a\rho}}\left(\int_{v}^\infty(x+y)^{\a\rho}f(y)\D y+\int_{-x}^{-v}(x+y)^{\a\rho}f(y)\D y\right).
\end{align*}
Recall that $f(y)\leq c|y|^{-\a-1}$ as $y\to\pm\infty$, and hence the first integral is upper bounded by
\[2^{\a\rho}c\int_x^\infty y^{\a\rho-\a-1}\D y+(2x)^{\a\rho}c\int_{v}^x y^{-\a-1}\D y\leq c x^{\a\rho}v^{-\alpha}\]
and the second has a similar bound. The result now follows.
\end{proof}

The following is an immediate consequence of the Doob's $h$-transform representation of the kernel; here $h(x)=x^{\a\rho}$.
\begin{lemma}\label{lem:conditioning}
For any $B\in\mathcal F_1$ it holds that 
\[\p_x^\uparrow(B,X_1\in\D y)=\p_x^\uparrow(X_1\in\D y)\p_x(B|\underline X_1>0,X_1=y)\]
\end{lemma}
\begin{proof}
For $0<t_1<\cdots<t_k<1$ we have
\begin{align*}
&\p_x^\uparrow(X_{t_1}\in\D x_1,\ldots,X_{t_k}\in\D x_k,X_1\in\D y)\\
&=\frac{h(x_1)}{h(x)}\p_x(X_{t_1}\in\D x_1,\underline X_{t_1}>0)\times\cdots\times\frac{h(y)}{h(x_k)}\p_{x_k}(X_{1-t_k}\in\D y,\underline X_{1-t_k}>0)\\
&=\frac{h(y)}{h(x)}\p_x(X_{t_1}\in\D x_1,\ldots,X_{t_k}\in\D x_k,X_1\in\D y,\underline X_1>0)\\
&=\p_x^\uparrow(X_1\in\D y)\p_x(X_{t_1}\in\D x_1,\ldots,X_{t_k}\in\D x_k|X_1=y,\underline X_1>0)
\end{align*}
and the result follows.
\end{proof}

\begin{lemma}\label{lem:sup}
There exists $c>0$ such that for all $x>v>0$ we have
\begin{align*}
\p^\uparrow_x(\overline X_1-x>v)<cv^{-\a},\qquad \p^\uparrow_x(x-\underline X_1>v)<cv^{-\a}.
\end{align*}
\end{lemma}
\begin{proof}
We only show the first statement, since the second follows the same arguments. 
According to Lemma~\ref{lem:conditioning} we find that 
\[\p^\uparrow_x(\overline X_1-x>v)=\int\p^\uparrow_x(X_1\in x+\D y)\p_x(\overline X_1-x>v|\underline X_1>0,X_1=x+y).\]
We may restrict the integration to the interval $[-v/2,v/2]$ in view of Lemma~\ref{lem:cond_bound}.
Thus it is sufficient to establish that 
\[\p(\overline X_1>v|\underline X_1>-x,X_1=y)<c v^{-\a}\]
for all $x>v$ and $y\in[-v/2,v/2]$.
But the quantity on the left is upper bounded by
\[\overline F(v,y)/\p(\underline X_1>-x|X_1=y),\]
where $\overline F(v,y)\leq cv^{-\a}$ according to Lemma~\ref{lem:bounds}; for bounded $v$ the result is obvious. Finally, observe that $\p(\underline X_1>-x|X_1=y)$ is bounded away from~0; here we may use Lemma~\ref{lem:bounds} applied to the process $-X$. The proof is complete.
\end{proof}

\begin{proof}[Proof of Proposition~\ref{prop:fluct_bound}]
Observe that the quantity of interest is upper bounded by
\[\p_x^\uparrow(\overline X_2-x>v/2 \text{ or }x-\underline X_2>v/2).\]
Hence the bound follows from Lemma~\ref{lem:sup}, which also holds for time $2$ instead of~$1$; use e.g.\ self-similarity here.
\end{proof}

%\bibliographystyle{abbrv}
%\bibliography{minimax.bib}

\end{document}